\documentclass[a4paper,10pt]{article}
\usepackage{amsmath}
\usepackage{amsthm}
\usepackage{amsfonts}
\usepackage{amssymb}
\usepackage[left=2.3cm,top=2.2cm,right=2.3cm,bottom=2.2cm]{geometry}

\newtheorem{lem}{Lemma}[section]
\newtheorem{thm}{Theorem}[section]
\newtheorem{co}[thm]{Corollary}
\numberwithin{equation}{section}

\begin{document}
\title{Branching random walk with a  random environment in time }
\author{Chunmao HUANG$^{a}$, Quansheng LIU$^{b,}$\footnote{Corresponding author at:  Universit\'e de Bretagne-Sud,  UMR 6205, LMBA, F-56000 Vannes, France.  \newline \indent \ \ Email addresses: cmhuang@hitwh.edu.cn (C. Huang), quansheng.liu@univ-ubs.fr (Q. Liu).}
\\
\small{\emph{$^{a}$Harbin institute of technology at Weihai, Department of mathematics, 264209, Weihai, China}}\\
\small{\emph{$^{b}$Universit\'e de Bretagne-Sud,  UMR 6205, LMBA, F-56000 Vannes, France}}}
 \date{}
\maketitle

\begin{abstract}
  We consider a branching random walk on $\mathbb{R}$ with a stationary and ergodic environment $\xi=(\xi_n)$ indexed by time $n\in\mathbb{N}$.  Let $Z_n$ be the counting measure of particles of generation $n$.  For the case where the corresponding branching process  $\{Z_n(\mathbb{R})\}$ $ (n\in\mathbb{N})$ is supercritical, we establish large deviation principles, central limit theorems and a local limit theorem  for the sequence of counting measures $\{Z_n\}$, and prove that the position $R_n$ (resp. $L_n$) of rightmost (resp. leftmost) particles of generation $n$ satisfies a law of large numbers.
 \\*


\end{abstract}
\section{Introduction}\label{LTS1}

\noindent A \emph{random environment in time} is modeled as a stationary and ergodic
sequence  of random variables, $\xi_n$, indexed by the time $n\in\mathbb{N}$, taking values in some measurable space $\Theta$. Each realization of $\xi_n$ corresponds to a distribution $\eta_n=\eta(\xi_n)$
  on $\mathbb{N}\times\mathbb{R}\times\mathbb{R}\times\cdots$.

When the environment $\xi=(\xi_n)$ is given, the process can be described as follows. At time $0$, there is an initial particle $\emptyset$ of generation $0$ located at $S_{\emptyset}=0\in\mathbb{R}$;
at time $1$, it is replaced by $N=N({\emptyset})$ particles of generation $1$, located at $L_i=L_i({\emptyset })$, $1\leq i\leq N$, where the random vector
$X({\emptyset})=(N, L_1, L_{2},\cdots)\in\mathbb{N}\times\mathbb{R}\times\mathbb{R}\times\cdots$ is of distribution $\eta_0=\eta(\xi_0)$  (given the environment $\xi$).   In general, each particle $u=u_1\cdots u_n$ of generation $n$ located at $S_u$ is replaced at
time {n+1} by $N(u)$ new particles $ui$ of generation $n+1$, located at
$$S_{ui}=S_u+L_i(u)\qquad(1\leq i\leq N(u)),$$
where the  random vector $X(u)=(N(u),L_1(u),L_2(u),\cdots)$ is of distribution $\eta_n=\eta(\xi_n)$.  Note that the values $L_i(u)$ for $i> N(u)$ do not play any role for our model; we introduce them only for convenience. We can for example take $L_i(u)=0$ for $i> N(u)$. All particles behave independently conditioned on the environment $\xi$.

Let $(\Gamma,  \mathbb{P}_\xi)$ be the probability space under which the
process is defined when the environment $\xi$ is fixed.  As usual,
$ \mathbb{P}_\xi$ is called {\em quenched law}.
The  total probabilifty space can be formulated as the product space
$(\Gamma\times \Theta^\mathbb{N },  \mathbb{P})$, where $ \mathbb{P} =  \mathbb{P}_{\xi}\otimes \tau $ in the sense  that for all measurable and positive
$g$, we have $$\int g d\mathbb{P} = \int_{\Theta^{\mathbb{N}}} \left(\int_\Gamma g(\xi,y) d \mathbb{P}_{\xi}(y) \right)d \tau  (\xi),$$
where $\tau$ is the law of the environment $\xi$.  The total probability $\mathbb{P}$ is usually called
{\em annealed law}. The quenched law $ \mathbb{P}_\xi$ may be considered to be the conditional
probability of $\mathbb{P}$ given $\xi$.  Let $\mathbb{U}=\{\emptyset\}\bigcup_{n\geq1}\mathbb{N}^n$ be the set of all finite sequence $u=u_1\cdots u_n$. By definition, under $\mathbb{P}_\xi$, the random vectors $\{X_u\}$, indexed  by $u\in\mathbb{U}$, are independent of each other, and each  $X_u$ has distribution $\eta_n=\eta(\xi_n)$ if $|u|=n$, where $|u|$ denotes the length of $u$.

Let $\mathbb{T}$ be the Galton-Watson tree  with defining element $\{N(u)\}$. We have: (a) $\emptyset\in\mathbb{T}$; (b) if $u\in\mathbb{T}$, then $ui\in\mathbb{T}$ if and only if $1\leq i\leq N(u)$; (c) $ui\in\mathbb{T}$ implies $u\in\mathbb{T}$. Let  $\mathbb{T}_n=\{u\in\mathbb{T}: |u|=n\}$ be the set of particles of generation $n$ and
$$Z_n=\sum_{u\in\mathbb{T}_n}\delta_{S_u}$$
be the counting measure of particles of generation $n$, so that for a subset $A$ of $\mathbb{R}$, $Z_n(A)$ is the number of particles of generation $n$ located in $A$. For any finite sequence $u$, let
$$ X^{(u)} = \sum_{i=1}^{N(u)}   \delta_{L_i (u)}   $$
be the counting measure  corresponding to the random vector $X(u)$, whose increasing points are $L_i(u)$,  $1\leq i \leq N(u)$.
Denote $$X_n=X^{(u_0|n)},$$ where $u_0=(1,1,\cdots)$ and $u_0|n$ is the restriction to its first $n$ tems, with the convention that   $u_0|0=\emptyset$.  For simplicity, we introduce the following notations:
\begin{equation}\label{LTE1.1}
N_n=X_n(\mathbb{R}),\qquad m_n= \mathbb{E}_\xi N_n,
 \qquad P_0=1\quad \text{and}\quad P_n=\mathbb{E}_\xi Z_n(\mathbb{R})=\prod_{i=0}^{n-1}m_i.
\end{equation}

Let
$$\mathcal{F}_0=\sigma(\xi),\quad
\mathcal{F}_n = \sigma ( \xi, (X(u): |u| <n)) \;\;\text{for $n\geq 1$}$$
 be the $\sigma$-field containing all the
information concerning the first $n$ generations. It is well known that the sequence
$\{Z_n(\mathbb{R})/P_{n}\}$ is a non-negative martingale under $\mathbb{P}_\xi$ for every $\xi$ with respect
to  the filtration $\mathcal {F}_n$, hence it converges almost surely (a.s.) to a random variable
denoted by $W$.  Throughout this paper we  always assume that
\begin{equation}\label{LTE1.2}
\mathbb{E}\log m_{0}\in (0,\infty )\quad \text{ and }\quad
\mathbb{E}\frac{N}{m_{0}}\log^{+}N<\infty.
\end{equation}
The first condition means that the corresponding branching process in random environment, $\{Z_n(\mathbb{R})\}$, is \emph{supercritical}; the second implies that $W$ is non-degenerate. Hence (see e.g. Athreya and Karlin (1971, \cite{a1}))
\begin{equation*}
\mathbb{P}_{\xi }(W>0)=\mathbb{P}_{\xi }(Z_{n}(\mathbb{R})\rightarrow \infty )=\lim_{n\rightarrow
\infty }\mathbb{P}_{\xi }(Z_{n}(\mathbb{R})>0)>0\quad a.s..
\end{equation*}

In  this paper,  we  are interested in asymptotic properties of the sequence of measures $\{Z_n\}$.

Our first objective is to show a large deviation principle for  $\{Z_n(n\cdot)\}$  (Theorem \ref{LTT2.1.5}). Our approach uses the G\"artner-Ellis theorem.  In the proof,
 we first demonstrate that the sequence of quenched means $\{\mathbb{E}_\xi Z_n(n\cdot)\}$ satisfies a large deviation principle,
 and  then  show that the free energy $\frac{\log\tilde Z_n(t)}{n}$, where $\tilde Z_n(t)=\sum_{u\in\mathbb{T}_n}e^{tS_u}$  denotes the partition function, converges a.s.  to a limit that we calculate explicitly (Theorem \ref{LTT2.1.4}). Moreover, we also show that the position $R_n$ (resp. $L_n$) of  rightmost (resp. leftmost) particles of generation $n$ satisfies a law of large numbers  (Theorem \ref{LTT2.1.7}): $\frac{R_n}{n}$ (resp. $\frac{L_n}{n}$) converges a.s.  to a limit that we determine explicitly. These results generalize those of Biggins (1977, \cite{biggins}), Franchi (1995, \cite{franchi}) and  Chauvin \& Rouault (1997, \cite{chauvin}) for the deterministic environment case.

 Our second objective is to show central limit theorems and related results for $\{Z_n \}$.
For a deterministic branching random walk, Kaplan and Asmussen (1976, \cite{ka}) proved the following central limit theorem. Assume that $m=\mathbb{E}N\in(1,\infty)$ and that  $\frac{\mathbb{E}X_0(\cdot)}{m}$ has mean $0$ and variance $1$. If $\mathbb{E}N(\log N)^{1+\varepsilon}<\infty$ for some $\varepsilon>0$, then
\begin{equation}\label{BREE1}
m^{-n}Z_n(-\infty, \sqrt{n}x]\rightarrow\Phi(x)W\qquad a.s.\quad \forall x\in\mathbb{R},
\end{equation}
where $\Phi(x)$ is the distribution function of the standard normal distribution $\mathcal{N}(0,1)$. They also gave a local version of (\ref{BREE1}) under the stronger moment condition that $\mathbb{E}N(\log N)^{\gamma}<\infty$ for some $\gamma>3/2$. The formule (\ref{BREE1}), which was  first conjectured by Harris \cite{Harris}, has been studied by many authors, see e.g. Stam (1966, \cite{stam}), Kaplan \& Asmussen (1976, \cite{ka}),  Klebaner (1982, \cite{k}) and Biggins (1990, \cite{b}). We shall show the following version of (\ref{BREE1})  (Theorem \ref{LTT1.1.6})  for a branching random walk in a random environment:  under certain moment conditions, the sequence of probability measures $\frac{Z_n(b_n\cdot+a_n)}{Z_n(\mathbb{R})}$, with $(a_n, b_n)$ that we calculate explicitly,  converges to the standard normal distribution $\mathcal{N}(0,1)$  in law a.s. on the survival event $\{Z_n\rightarrow\infty\}$.  The technic in the proof is a mixture of  Klebaner (1982) and Biggins (1990) by  considering the characteristic function and  choosing an appropriate truncation function. We shall also show a corresponding local limit theorem (Theorem \ref{LTT1.1.9}) under stronger moment conditions, which generalizes the result of Biggins (1990, Theorem 7) on deterministic branching random walks. From Theorem \ref{LTT1.1.9} we obtain another form of local limit theorem (Corollary \ref{BRCC1}), which coincides with the result of Kaplan \&  Asmussen (1976, Theorem 2) for the deterministic environment case.

Moreover, we shall also show large deviation principles and central limit theorems for probability mesures with different normins: ${\frac{\mathbb{E}_\xi Z_n(\cdot)}{\mathbb{E}_\xi Z_n(\mathbb{R})}}$, $ {\frac{\mathbb{E} Z_n(\cdot)}{\mathbb{E} Z_n(\mathbb{R})}}$, $ {\mathbb{E}\frac{Z_n(\cdot)}{\mathbb{E}_\xi Z_n(\mathbb{R})}}$ and $\mathbb{E}_\xi\frac{Z_n(\cdot)}{Z_n(\mathbb{R})}$.
\\*

This paper is organized as follows. In Sections \ref{LTS2} - \ref{LTS5}, we consider   large deviations. In Section \ref{LTS2}, we show large deviation principles for $ {\mathbb{E}_\xi Z_n(n\cdot)}$, $ {\mathbb{E}Z_n(n\cdot)}$ and $ {\mathbb{E}\frac{Z_n(n\cdot)}{\mathbb{E}_\xi Z_n(\mathbb{R})}}$. In Section \ref{LTS3}, we state a convergence result for the free energy, a  large deviation principle for $Z_n(n\cdot)$ and laws of large numbers for $R_n$ and $L_n$. In Section \ref{LTS4}, we prove the results of Section \ref{LTS3}. In Section \ref{LTS5}, we show a large deviation principle for $\mathbb{E}_\xi\frac{Z_n(n\cdot)}{Z_n(\mathbb{R})}$. In Sections \ref{LTS6} - \ref{LTS12}, we study central limit theorems. In Section \ref{LTS6}, we consider a branching random walk in a varying environment and state the corresponding limit theorems. In Sections \ref{LTS7} and \ref{LTS8}, we prove the results of Section \ref{LTS6}. From Sections \ref{LTS9} to \ref{LTS12}, we return to a branching random walk in a random environment:  in Section \ref{LTS9}, we show central limit theorems for  $ {\frac{\mathbb{E}_\xi Z_n(\cdot)}{\mathbb{E}_\xi Z_n(\mathbb{R})}}$, $ {\frac{\mathbb{E} Z_n(\cdot)}{\mathbb{E} Z_n(\mathbb{R})}}$ and $ {\mathbb{E}\frac{Z_n(\cdot)}{\mathbb{E}_\xi Z_n(\mathbb{R})}}$; in Section \ref{LTS10}, we state a central limit theorem and a local limit theorem for $\frac{Z_n(\cdot)}{Z_n(\mathbb{R})}$, which  are proved in Section \ref{LTS11}; in Section \ref{LTS12}, we show central limit theorems for $\mathbb{E}_\xi\frac{Z_n(\cdot)}{Z_n(\mathbb{R})}$ and $\mathbb{E}\frac{Z_n(\cdot)}{Z_n(\mathbb{R})}$.

%%%%%%%%%%%%%%%%%%%%%%%%% section 2 %%%%%%%%%%%%%%%%%%%%%%%%%
\section{Large deviations for $ {\mathbb{E}_\xi Z_n(n\cdot)}$, $ {\mathbb{E}Z_n(n\cdot)}$ and $ {\mathbb{E}\frac{Z_n(n\cdot)}
{\mathbb{E}_\xi Z_n(\mathbb{R})}}$}\label{LTS2}
\noindent To study large deviations of $Z_n$, we begin with the  study of its  quenched and annealed means. For $n\in \mathbb{N}$ and $t\in \mathbb{R}$, let
\begin{equation}
m_n(t) := \mathbb{E}_\xi\int e^{itx}X_n(dx)=  \mathbb{E}_\xi \sum_{i=1}^{N(u)} e^{tL_i(u)} \quad(u\in\mathbb{T}_n),
\end{equation}
be the Laplace transform of the counting measure describing the
evolution of the system at time $n$. In particular,
$$ m_0 (t) =  \mathbb{E}_\xi \sum_{i=1}^{N} e^{tL_{i}}, \quad m_0(0) =  \mathbb{E}_\xi N=m_0. $$
We assume that
\begin{equation} \label{H1}
 E|L_1| < \infty, \;  \mathbb{E}|\log m_0 (t)|<\infty \mbox{ and } \mathbb{E}|\frac{m_0' (t)}{m_0 (t)} |<\infty
\end{equation}
for all $t\in \mathbb{R}$. The last two moment conditions imply  that
$$\Lambda (t) := \mathbb{E}\log m_0 (t)  \mbox{ and }  \Lambda' (t) : =  \mathbb{E} \frac{m_0' (t)}{m_0 (t)}   $$
are well defined as real numbers, that $\Lambda (t)$
is differentiable everywhere on $\mathbb{R}$ with $\Lambda' (t)$ as its derivative (this can be easily verified by the dominated convergence theorem, using the fact that the function $t \mapsto  \frac{m_0' (t)}{m_0 (t)} $ is increasing).
Let
\begin{equation*}
\Lambda^* (x) = \sup_{t\in \mathbb{R}} \{ xt -  \Lambda (t)\}
\end{equation*}
be the Legendre transform of $\Lambda$. Then  \begin{equation*}
 \Lambda^* (x) = \left \{
   \begin{array}{lll}
          t\Lambda'(t) - \Lambda (t)  & if & x= \Lambda'(t) \mbox{ for some } t\in
          \mathbb{R},  \\
         +\infty  & if & x\geq\Lambda '(+\infty) \mbox{ or } x\leq\Lambda
         '(-\infty), \\
   \end{array}
     \right.
\end{equation*}
and
$$ \min_x \Lambda^* (x) =  \Lambda^* (\Lambda'(0)) = -\Lambda (0) = -\mathbb{E}\log m_0 <0.$$

With these notations, now we can state our large deviation principle  for the quenched means $\mathbb{E}_\xi Z_n(n\cdot)$, which will leads to a large deviation principle about $Z_n(n\cdot)$.

\begin{thm}[Large deviation principle for quenched means $\mathbb{E}_\xi Z_n(n\cdot)$]\label{LTT2.1.1} Assume (\ref{H1}).
For almost every $\xi$, the sequence of finite measures $A \mapsto
 \mathbb{E}_\xi Z_n (nA)$ satisfies a principle of large deviation with rate
function $\Lambda^* $: for each measurable subset $A$ of $\mathbb{R}$,
\begin{eqnarray*}\label{LDP1}
   - \inf_{x\in A^o} \Lambda^* (x)
   &\leq& \liminf_{n\rightarrow \infty}
             \frac{1}{n}  \log  \mathbb{E}_\xi Z_n (nA) \\
  & \leq &  \limsup_{n\rightarrow \infty}
            \frac{1}{n} \log  \mathbb{E}_\xi Z_n (nA)
   \leq - \inf_{x\in \bar A}  \Lambda^*(x), \\
\end{eqnarray*}
where $A^o$ denotes the interior of $A$, and $\bar A$ its closure.
\end{thm}

\begin{proof}[Proof]%[Proof of Theorem \ref{LTT2.1.1}]
Notice that the measures  $q_n (\cdot) = \mathbb{E}_\xi Z_n(\cdot)$ satisfy
\begin{equation*}
  \tilde {q}_n (t) := \int e^{tx} q_n (dx) =\mathbb{ E}_\xi \sum_{u\in\mathbb{T}_n}
  e^{t S_u} = m_0 (t) ...m_{n-1} (t).
\end{equation*}
By the ergodic theorem,
\begin{equation*}
   \lim_{n\rightarrow \infty}  \frac{1}{n} \log  \tilde {q}_n (t) =
   \Lambda (t) := \mathbb{E}\log m_0 (t) \quad \mbox{ a.s.. }
\end{equation*}
Therefore, applying the G\"artner-Ellis theorem (\cite{z}, p.53, Exercise 2.3.20) to the sequence of normalized
probability measures $q_n(n\cdot)/q_n(\mathbb{R})$, we obtain the desired result.
\end{proof}

If the environment  is $i.i.d.$, similar results can be established for annealed means. Let
 $$\Lambda_a (t) = \log \mathbb{E} m_0 (t),$$
 and $\Lambda_a^* $ be its Legendre transform. Then we have:

\begin{thm}[Large deviation principle for annealed means $\mathbb{E}Z_n(n\cdot)$]\label{LTT2.1.2}
Assume that $\xi_n$ are i.i.d..   If $\mathbb{E}m_0(t)\in(0,\infty)$  for all $t\in\mathbb{R}$, then the sequence of finite measures
$A \mapsto  \mathbb{E}Z_n (nA)$ satisfies a principle of large deviation with
rate function $\Lambda_a^* $: for each measurable subset $A$ of
$\mathbb{R}$,
\begin{eqnarray*}\label{LDP2}
   - \inf_{x\in A^o} \Lambda_a^* (x)
   &\leq& \liminf_{n\rightarrow \infty}
             \frac{1}{n}  \log  \mathbb{E}Z_n (nA) \\
  & \leq &  \limsup_{n\rightarrow \infty}
            \frac{1}{n} \log  \mathbb{E}Z_n (nA)
   \leq - \inf_{x\in \bar A}  \Lambda_a^*(x), \\
\end{eqnarray*}
where $A^o$ denotes the interior of $A$, and $\bar A$ its closure.
\end{thm}

\noindent \textbf{Remark.}  It is easy to see that
$$\Lambda_a (t) \geq \Lambda (t) \quad \mbox{ and } \quad \Lambda_a^* (x) \leq \Lambda^* (x).$$
 \\*

\begin{proof}[Proof of Theorem \ref{LTT2.1.2}]
The proof is similar to that of Theorem \ref{LTT2.1.1}, with
 $q_n (\cdot) = \mathbb{E }Z_n(\cdot)$. Notice  that  when $\xi_n$ are i.i.d.,
\begin{equation*}
  \tilde {q}_n (t) := \int e^{tx} q_n (dx) = \mathbb{E}\sum_{u\in\mathbb{T}_n}
  e^{t S_u} =\left( \mathbb{E}m_0 (t) \right)^n.
\end{equation*}
\end{proof}

If we consider the measures $\mathbb{E}\frac{Z_n(\cdot)}{ \mathbb{E}_\xi Z_n(\mathbb{R})}$ instead of $\frac{ \mathbb{E}Z_n(\cdot)}{ \mathbb{E}Z_n(\mathbb{R})}$,
we can obtain another large deviation principle.

\begin{thm}[Large deviation principle for $\mathbb{E}\frac{Z_n(n\cdot)}{ \mathbb{E}_\xi Z_n(\mathbb{R})}$]\label{LTT2.1.3}
Assume that $\xi_n$ are i.i.d..  Let
 $\bar{\Lambda}_a (t) = \log \mathbb{E}\frac{m_0 (t)}{m_0}$
 and $\bar{\Lambda}_a^* $ be its Legendre transform.  If $\mathbb{E}\frac{m_0(t)}{m_0}\in(0,\infty)$ for all $t\in\mathbb{R}$, then the sequence of finite measures
$A \mapsto \mathbb{E}\frac{Z_n(nA)}{ \mathbb{E}_\xi Z_n(\mathbb{R})}$ satisfies a principle of large deviation with
rate function $\bar{\Lambda}_a^* $: for each measurable subset $A$ of
$\mathbb{R}$,
\begin{eqnarray*}\label{LDP3}
   - \inf_{x\in A^o} \bar{\Lambda}_a^*  (x)
   &\leq& \liminf_{n\rightarrow \infty}
             \frac{1}{n}  \log \mathbb{E}\frac{Z_n(nA)}{ \mathbb{E}_\xi Z_n(\mathbb{R})} \\
  & \leq &  \limsup_{n\rightarrow \infty}
            \frac{1}{n} \log \mathbb{E}\frac{Z_n(nA)}{ \mathbb{E}_\xi Z_n(\mathbb{R})}
   \leq - \inf_{x\in \bar A}  \bar{\Lambda}_a^*(x), \\
\end{eqnarray*}
where $A^o$ denotes the interior of $A$, and $\bar A$ its closure.
\end{thm}

\begin{proof}[Proof]%[Proof of Theorem \ref{LTT2.1.3}]
The proof is still similar to that of Theorem \ref{LTT2.1.1}, with  $q_n (\cdot) = \mathbb{E }\frac{Z_n(\cdot)}{\mathbb{E}_\xi Z_n(\mathbb{R})}$ whose Laplace transform is
\begin{equation*}
  \tilde {q}_n (t) := \int e^{tx} q_n (dx) =  \left( \mathbb{E}\frac{m_0 (t)}{m_0} \right)^n.
\end{equation*}
\end{proof}

%%%%%%%%%%%%%%%%%%%%%%%%%%%%%     section 3            %%%%%%%%%%%%%%%%%%%%%%%%%%%%%%%%%%%%%
\section{Convergence of the free energy; large deviations for $ {Z_n(n\cdot)}$; positions of rightmost and leftmost particles}\label{LTS3}
Now we consider large deviations for the sequence of measures $ \{Z_n(n\cdot)\}$.
Let
\begin{equation}
\tilde Z_n (t) : = \int e^{tx}Z_n (dx)  = \sum_{u\in\mathbb{T}_n}
 e^{tS_u}
 \end{equation}
 be the Laplace transform of $Z_n$, also called \emph{partition
 function} by physicians. We are interested in the convergence of the  \emph{free energy} $
\frac{\log \tilde Z_n(t)}{n}$. To this end we define two critical values $t_-$ and $t_+$.   Let
$$ \rho (t) = t\Lambda'(t) - \Lambda (t), \qquad t\in \mathbb{R}. $$
Notice that $ \rho'(t) = t\Lambda''(t)$. Therefore
 $\rho (t)$ decreases on $(-\infty,0]$, increases on $[0,\infty)$, and attains
 its minimum at $0$:
 $$ \min_t \rho (t) = \rho (0) = - \Lambda (0) <0.$$
Let
$$ t_-= \inf \{ t \in \mathbb{R}: t\Lambda'(t) - \Lambda (t) \leq 0\}, $$
$$t_+= \sup \{ t \in \mathbb{R}: t\Lambda'(t) - \Lambda (t) \leq 0\}. $$
Then $ -\infty \leq t_- <0 < t_+ \leq \infty$, $t_-$ and $t_+$ are
two solutions of $t\Lambda'(t) - \Lambda (t) = 0$ if they are
finite.
%%
%For $c>0$,  set
%\begin{equation}
%m_0^c(t)=\mathbb{E}_\xi \sum_{i=1}^{N\wedge c} e^{tLi},
%\end{equation}
%where and throughout we write $a\wedge b=\min(a,b)$.
%Assume that    $\mathbb{E}\log^-m_0^c(t)<\infty$ for some $c>0$.
%% Q: we use the more simple condition E |L_1| < \infty
%%
%% Note that this assumption always holds for deterministic environment case.
For simplicity, we also assume that
\begin{equation}
N\geq1\qquad a.s.,
 \end{equation}
so that $Z_n(\mathbb{R})\rightarrow\infty$ a.s..

\begin{thm}[Convergence of the free energy]\label{LTT2.1.4}
It is a.s.  that for all $t\in\mathbb{R}$,
\begin{equation}\label{LTE3.3}
\lim_{n\rightarrow \infty} \frac{\log \tilde Z_n(t)}{n} =
\tilde\Lambda (t) :=
   \left\{
   \begin{array} {lll}
          \Lambda (t)  & if &   t\in (t_-, t_+) ,\\
            t\Lambda'(t_+) & if & t\geq t_+, \\
            t\Lambda'(t_-) & if & t\leq t_-.
   \end{array}
     \right.
\end{equation}
\end{thm}
For the deterministic environment case, see Chauvin \&  Rouault
(1997, \cite{chauvin}) and Franchi (1995, \cite{franchi}).

Let $\tilde\Lambda^*(x)$ be the Legendre transform of
$\tilde\Lambda(t)$.  By Theorem \ref{LTT2.1.4} and the G\"artner- Ellis' theorem, we
 immediately obtain the following large deviation principe for $Z_n(n\cdot)$.

\begin{thm}[Large deviation principle for $Z_n(n\cdot)$]\label{LTT2.1.5}
It is a.s.   that the sequence of finite measures
$A \mapsto  {Z_n(nA)}$ satisfies a principle of large deviation with
rate function $\tilde{\Lambda}^* $: for each measurable subset $A$ of
$\mathbb{R}$,
\begin{eqnarray*}\label{LDP1}
   - \inf_{x\in A^o} \tilde\Lambda^* (x)
   &\leq& \liminf_{n\rightarrow \infty}
             \frac{1}{n}  \log Z_n (nA) \\
  & \leq &  \limsup_{n\rightarrow \infty}
            \frac{1}{n} \log   Z_n (nA)
   \leq - \inf_{x\in \bar A}  \tilde\Lambda^*(x), \\
\end{eqnarray*}
where $A^o$ denotes the interior of $A$, and $\bar A$ its closure.
\end{thm}

\noindent \textbf{Remark.}   It  can be seen that $\tilde\Lambda(t)\leq\Lambda(t)$, so that $\tilde\Lambda^*(x)\geq\Lambda^*(x)$. Moreover,
  \begin{equation*}
 \tilde\Lambda^* (x) = \left \{
   \begin{array}{lll}
         \Lambda^* (x) & if & x\in [\Lambda'(t_-) , \Lambda' (t_+)],  \\
         +\infty  & if & x<\Lambda '(t_-) \mbox{ or } x>\Lambda
         '(t_+), \\
   \end{array}
     \right.
\end{equation*}

\begin{co}\label{LTC2.1.6}
It is a.s.   that
$$ \lim_{n \rightarrow\infty} \frac {1}{n}  \log Z_n [nx, \infty)
 = -\Lambda^*(x) > 0 \mbox{ if } x \in (\Lambda'(0),
 \Lambda'(t_+)), $$
$$ \lim_{n \rightarrow\infty} \frac {1}{n}  \log Z_n (-\infty, nx]
 = -\Lambda^*(x) > 0 \mbox{ if } x \in (\Lambda'(t_-), \Lambda'(0)).$$
\end{co}
For deterministic branching random walks, see Biggins (1977, \cite{biggins})  and Chauvin \&  Rouault (1997, \cite{chauvin}).

\noindent \textbf{Remark.}  $$ x \in (\Lambda'(0),
 \Lambda'(t_+)) \mbox{ if and only if } x> \Lambda'(0) \mbox{ and } \Lambda^*(x) <
 0. $$
$$ x \in (\Lambda'(t_-),
 \Lambda'(0))  \mbox{ if and only if } x< \Lambda'(0) \mbox{ and } \Lambda^*(x) <
 0. $$

 If the set $\mathbb{T}_n\neq \emptyset$, let
$$ L_n = \min_{u\in\mathbb{T}_n} S_u   \quad (resp.\quad R_n = \max_{u\in\mathbb{T}_n} S_u)$$
be the position of leftmost (resp. rightmost) particles of generation $n$. We can see that  $L_n $ (resp. $R_n $) satisfies a law of large numbers.

  \begin{thm}[Asymptotic properties of $L_n$ and $R_n$]\label{LTT2.1.7}
 It is a.s.  that
 $$ \lim_{n \rightarrow\infty}\frac{L_n}{n} = \Lambda'(t_-), $$
$$ \lim_{n \rightarrow\infty} \frac{R_n}{n} = \Lambda'(t_+). $$
\end{thm}

For deterministic  branching random walks, see Biggins (1977) and Chauvin \& Rouault (1997).

%%%%%%%%%%%%%%%%%%%%%%%%%%%           section   4              %%%%%%%%%%%%%%%%%%%%%%%%%%%%%%%%%%%%%
\section{Proofs of Theorems \ref{LTT2.1.4} and  \ref{LTT2.1.7}}\label{LTS4}
Let us give the proofs of Theorems \ref{LTT2.1.4} and  \ref{LTT2.1.7} which are composed by some lemmas. Similar arguments have been used in Franchi (1995, \cite{franchi}) and Chauvin \& Rouault (1997).

Observe that
$$ W_n(t) := \frac{\tilde Z_n (t)}{\mathbb{E}_\xi {\tilde Z_n (t)}}
= \frac{ \sum_{u\in\mathbb{T}_n} e^{tS_u}}{m_0(t)...m_{n-1}(t)}   $$ is a
martingale, therefore it converges a.s. to a random variable  $W(t) \in
[0,\infty)$. In the deterministic environment case, this martingale has been
studied by Kahane \&  Peyri\`ere (1976), Biggins (1977),
Durrett \& Liggett (1983), Guivarc'h (1990), Lyons (1997) and
Liu (1997, 1998, 2000, 2001), etc. in different contexts.

 The following lemma concerns the non degeneration of $W(t)$.

\begin{lem}\label{LTL2.3.1}
 If $t\in (t_-,t_+)$ and $\mathbb{E}W_1(t)\log^+ W_1(t) <\infty$, then
  $$ W(t) >0 \quad {\mbox  a.s. } $$
 If $t\leq t_-$ or $t\geq t_+$, then
$$ W(t) = 0 \quad {\mbox  a.s. } $$
\end{lem}

Notice that $t\in(t_-,t_+)$  is equivalent to $t\Lambda'(t)-\Lambda(t)<0$. Therefore the lemma is an immediate consequence of
Theorem 7.2 of Biggins and Kyprianous (2004) on a branching process in a random environment, or of a result of
  Kuhlbusch (2004, \cite{ku}) on weighted branching
processes in random environment.

\begin{lem}\label{LTL2.3.2}
 If $t\in (t_-,t_+)$, then
\begin{equation}\label{LTE2.3.1}
\lim_{n\rightarrow\infty}\frac{1}{n}\log \tilde Z_n(t)=\Lambda(t) \mbox{ a.s..}
\end{equation}
\end{lem}

\begin{proof}[Proof]
If $\mathbb{E}W_1(t)\log^+ W_1(t) <\infty$, by Lemma \ref{LTL2.3.1}, $W(t)>0$ a.s.. Consequently,
$$
\frac{1}{n}\log \tilde Z_n(t)=\frac{1}{n}\log W_n(t)+\frac{1}{n}\sum_{i=0}^{n-1}\log m_i(t)\rightarrow \mathbb{E}\log m_0(t)=\Lambda(t)\;\;a.s..
$$

 We now consider the general case where $\mathbb{E}W_1(t)\log^+ W_1(t) $ may be infinite.
%Assume that
% $\mathbb{E}\log^-m_0^{c_0}(t)<\infty$. Let $c>0$.
We only consider the case where $t \in [0,t_+)$ (the case where  $t \in (t_-, 0])$ can be considered in a similar way, or by considering
$(-L_u)$ instead of   $(L_u)$).

For the lower bound, we use an truncating argument.  For $c \in \mathbb{N} $, we construct a new branching random walk  in a random environment (BRWRE)
using  $X^c(u)=(N(u)\wedge c, L_{1}(u), L_{2}(u), \cdots)$ instead of  $X(u)=(N(u), L_{1}(u), L_{2}(u), \cdots)$,
%with $N^c(u)=N(u)\wedge c$.
where and throughout we write $a\wedge b=\min(a,b)$.
We shall apply
Lemma \ref{LTL2.3.1} to the new BRWRE. We define   $m_n^c (t)$, $W_n^c(t)$, $\Lambda_c(t)$ and $t_+^c$ for the new BRWRE just
as just as  $m_n (t)$
$W_n(t)$,   $\Lambda (t)$ and $t^c$  were defined for the original  BRWRE.
% For $c>0$,  set
% \begin{equation}
%  m_0^c(t)=\mathbb{E}_\xi \sum_{i=1}^{N\wedge c} e^{tLi},
% \end{equation}
%
% Assume that    $\mathbb{E}\log^-m_0^c(t)<\infty$ for some $c>0$.
%

We first show that $\Lambda_c(t):=\mathbb{E}\log m_0(t)\uparrow\Lambda(t)$ as $c\uparrow \infty$.
Clearly, $m_0^c(t)=\mathbb{E}_\xi\sum_{i=1}^{N\wedge c}e^{tL_i}\uparrow m_0(t)$ as $c\uparrow \infty$. This leads to $\mathbb{E}\log^+m_0^c(t)\uparrow\mathbb{E}\log^+m_0(t) $ by the monotone convergence theorem.  On the other hand, for $c\geq 1$,
we have
$$ \log^-m_0^c(t) \leq \log^- m_0^{1}(t) = \log^-  \mathbb{E}_\xi e^{tL_1}  \leq  t \mathbb{E}_\xi |L_1 | $$
(as $\mathbb{E}_\xi e^{tL_1} \geq e^{- t \mathbb{E}_\xi  |L_1|}$  by Jensen's inequality).
Therefore by the condition
$E|L_1| < \infty$ and the dominated convergence theorem,
 $\mathbb{E}\log^-m_0^c(t)\downarrow\mathbb{E}\log^-m_0(t) $.
 % since $\mathbb{E}\log^-m_0^{c_0}(t)<\infty$.

We next prove that for $c>0$ large enough, $t \in [0, t_+^c)$,  which is equivalent to  $t\Lambda'_c(t)-\Lambda_c(t)<0$.
%the new BRWRE satisfies the conditions of Lemma \ref{LTL2.3.1}.
Recall that $t\in[0,,t_+)$  is equivalent to $t\Lambda'(t)-\Lambda(t)<0$.
 By the definition of $\Lambda'(t)$, there exists a $h>0$ such that
$$t\frac{\Lambda(t+h)-\Lambda(t)}{h}-\Lambda(t)<0.$$
Since $\Lambda_c  \uparrow\Lambda  $ as $c\uparrow\infty$, we have for $c$ large enough,
\begin{equation}\label{LTE2.3.2}
t\frac{\Lambda_c(t+h)-\Lambda_c(t)}{h}-\Lambda_c(t)<0.
\end{equation}
The convexity of $\Lambda_c(t)$ shows that
\begin{equation}\label{LTE2.3.3}
\Lambda'_c(t)\leq \frac{\Lambda_c(t+h)-\Lambda_c(t)}{h}.
\end{equation}
Combing (\ref{LTE2.3.3}) with (\ref{LTE2.3.2}) we obtain for $c$ large enough,
\begin{equation}\label{LTE2.3.3}
t\Lambda'_c(t)-\Lambda_c(t)<0.
\end{equation}
%If $t<0$, we can also obtain (\ref{LTE2.3.3}) by a similar argument.

 We finally prove that  $\mathbb{E}W^c_1(t)\log^+ W^c_1(t) <\infty$.
 Let $Y=W^c_1(t)$. we define a random variable $X$ whose distribution is determined by
$$\mathbb{E}_\xi g(X)=\mathbb{E}_\xi Yg(Y)$$
for all bounded and measurable function $g$ (notice that $\mathbb{E}_\xi Y = 1$ by definition).
 For $x\in\mathbb{R}$, let
\begin{equation*}
l(x)=\left\{	\begin{array}{ll}
				{x}/{e}& \text{if $ x<e$,}\\
				\log x &\text{if $x\geq e$.}
			\end{array}
			\right.
\end{equation*}
It is clear that $l$ is concave and $\log^+x\leq l(x)\leq 1+\log^+x$ for all $x\in\mathbb{R}$. Thus
\begin{eqnarray*}
\mathbb{E}_\xi Y\log ^+ Y=\mathbb{E}_\xi\log ^+X&\leq&\mathbb{E}_\xi l(x)\\
&\leq& l(\mathbb{E}_\xi X)
=l(\mathbb{E}_\xi Y^2)\\
&\leq&1+\log^+\mathbb{E}_\xi Y^2\\
&\leq&1+\log^+\left(\frac{cm_0^c(2t)}{m_0^c(t)^2}\right),
\end{eqnarray*}
where the last inequality holds as $ (\sum_{i=1}^{N \wedge c}  e^{tL_i})^2  \leq  (N \wedge c) \sum_{i=1}^{N \wedge c} e^{2 tL_i}$.
Taking expectation in the above inequality , we get
\begin{eqnarray*}
\mathbb{E}W^c_1(t)\log^+ W^c_1(t) =\mathbb{E}Y\log^+Y&\leq& 1+\mathbb{E}\log^+\left(\frac{cm_0^c(2t)}{m_0^c(t)^2}\right)\\
&\leq&1+\log c +\mathbb{E} \log^+m_0(2t) +2\mathbb{E}\log^-m_0^{c}(t)<\infty.
\end{eqnarray*}

We have therefore proved that for $c>0$ large enough, the new BRWRE satisfies the conditions of Lemma \ref{LTL2.3.1}, so that
$$
\lim_{n\rightarrow\infty}\frac{1}{n}\log \tilde Z^c_n(t)= \mathbb{E}\log m^c_0(t)=\Lambda_c(t)\;\;a.s..
$$
Notice that $\tilde Z_n(t)\geq \tilde Z_n^c(t)$. It follows that
$$\liminf_{n\rightarrow\infty}\frac{1}{n}\log \tilde Z_n(t)\geq\Lambda_c(t)\;\;a.s..$$
Letting $c\uparrow\infty$, we obtain
$$\liminf_{n\rightarrow\infty}\frac{1}{n}\log \tilde Z_n(t)\geq\Lambda(t)\;\;a.s..$$

For th upper bound, from the decomposition $\frac{1}{n}\log \tilde Z_n(t)=\frac{1}{n}\log W_n(t)+\frac{1}{n}\sum_{i=0}^{n-1}\log m_i(t)$
and the fact that $W_n (t) \rightarrow W(t) < \infty$ a.s., we obtain  that
$$\limsup_{n\rightarrow\infty}\frac{1}{n}\log \tilde Z_n(t)\leq\Lambda(t)\;\;a.s..$$
This completes the proof.
\end{proof}

\begin{lem}\label{LTL2.3.3}
It is a.s. that
$$ \limsup_{n\rightarrow\infty} \frac{R_n}{n} \leq \Lambda'(t_+) .$$
\end{lem}

\begin{proof}[Proof]
For $a> \Lambda'(t_+)$,  we have $\Lambda^*(a)>0$. By Theorem \ref{LTT2.1.1},
$$\lim_{n\rightarrow\infty}\frac{1}{n}\mathbb{E}_\xi Z_n [an,\infty) =-\Lambda^*(a)<0\;\;a.s..$$
This leads to  $\sum_n \mathbb{P}_\xi ( Z_n [an,\infty) \geq 1) <\infty$ a.s.. It follows that
by Borel-Cantelli's lemma,  $\mathbb{P}_\xi $ a.s. ,
$$Z_n[an,\infty)=0 \quad \mbox{ for $n$ large enough}.$$
Therefore $R_n<an$, so that a.s.,
$$\limsup_{n\rightarrow\infty} \frac{R_n}{n} \leq a.\;$$
Letting $a\downarrow \Lambda'(t_+)$, we obtain the desired result.
\end{proof}

\begin{lem}\label{LTL2.3.4}
 If $t\geq t_+$, then a.s.,
\begin{equation}\label{LTE2.3.5}
\lim_{n\rightarrow\infty} \frac{\log \tilde Z_n (t)}{n} = t\Lambda'(t_+) .
\end{equation}
\end{lem}

\begin{proof}[Proof]
For the upper bound, we only consider the case where $t_+ < \infty$.  Choose $0<t_0<t_+\leq t$. Since $S_u\leq R_n$ for $u\in\mathbb{T}_n$, we have
$$tS_u\leq t_0S_u+(t-t_0)R_n,$$
so that
$$\tilde Z_n(t)\leq\tilde Z_n(t_0)e^{(t-t_0)R_n}.$$
Thus
$$ \frac{\log\tilde Z_n(t)}{n}\leq\frac{\log\tilde Z_n(t_0)}{n}+(t-t_0)\frac{R_n}{n}.$$
Letting $n\rightarrow\infty$ and using Lemma \ref{LTL2.3.3}, we get a.s.,
$$ \limsup_{n\rightarrow\infty} \frac{\log \tilde Z_n (t)}{n} \leq \Lambda (t_0) +  (t -t_0) \Lambda'(t_+). $$
Letting $t_0\uparrow t_+$ and using $\Lambda (t_+)  -t_+
\Lambda'(t_+) =0$, we obtain a.s.,
$$\limsup_{n\rightarrow\infty} \frac{\log \tilde Z_n (t)}{n} \leq t\Lambda'(t_+).$$

For the lower bound, as $\log \tilde Z_n (t) $ is a convex function of $t$, for
$t_- < t_0 < t_1 < t_+ \leq t$, we have
$$   \frac{\log \tilde Z_n (t) - \log \tilde Z_n (t_0)}{t-t_0} \geq
    \frac{\log \tilde Z_n (t_1) - \log \tilde Z_n (t_0)}{t_1-t_0}. $$
Dividing the inequality by $n$ and applying  Lemma \ref{LTL2.3.2}
to $t_0$ and $t_1$, we obtain a.s.,
$$ \liminf_{n\rightarrow\infty} \frac{\log \tilde Z_n (t)}{n} \geq
 \Lambda (t_0) + \frac{t-t_0}{t_1-t_0}  (\Lambda (t_1) - \Lambda
 (t_0)). $$
Letting $t_1 \downarrow t_0$, we get a.s.,
$$  \liminf_{n\rightarrow\infty} \frac{\log \tilde Z_n (t)}{n}
  \geq   \Lambda (t_0) + (t-t_0) \Lambda'(t_0). $$
Letting $t_0 \uparrow t_+$ and using $\Lambda (t_+)  -t_+
\Lambda'(t_+) =0, $ we obtain  a.s.,
$$\liminf_{n\rightarrow\infty} \frac{\log \tilde Z_n (t)}{n} \geq t\Lambda'(t_+).$$
This completes the proof.
\end{proof}

\begin{lem}\label{LTL2.3.5}
It is a.s. that
$$ \liminf_{n\rightarrow\infty} \frac{R_n}{n} \geq \Lambda'(t_+) .$$
\end{lem}

\begin{proof}[Proof]
Notice that $S_u\leq R_n$ for $u\in\mathbb{T}_n$, we have
$$\tilde Z_n(t)\leq Z_n(\mathbb{R})e^{tR_n},$$
so that for each $0<t<\infty$,
\begin{equation}\label{LTE2.3.6}
\frac{\log\tilde Z_n(t)}{n}\leq\frac{ \log Z_n(\mathbb{R})}{n}+t\frac{R_n}{n}.
\end{equation}
If $t_+<\infty$, then by Lemma \ref{LTL2.3.4}, the above inequality gives for $t>t_+$, a.s.,
$$\Lambda'(t_+)\leq \frac{1}{t}\mathbb{E}\log m_0+\liminf_{n\rightarrow\infty} \frac{R_n}{n}.$$
Letting $t\uparrow\infty$, we obtain the desired result.  If $t_+=\infty$,  then by Lemma \ref{LTL2.3.2}, the  inequality (\ref{LTE2.3.6}) gives for $t>0$, a.s.,
$$\frac{\Lambda(t)}{t}\leq \frac{1}{t}\mathbb{E}\log m_0+\liminf_{n\rightarrow\infty} \frac{R_n}{n}.$$
Letting $t\uparrow\infty$, we get a.s.,
$$\liminf_{n\rightarrow\infty} \frac{R_n}{n}\geq \Lambda'(\infty)=\Lambda'(t_+).$$
\end{proof}

The conclusions for $t\leq t_-$ and $L_n$ can be obtained in a similar way, or by applying the obtained results for $t\geq t_+$ and $R_n$ to the opposite branching random walk $-S_u$.  Hence Theorem \ref{LTT2.1.7} holds, and (\ref{LTE3.3}) holds a.s. for each fixed $t\in\mathbb{R}$. So a.s. (\ref{LTE3.3}) holds for all rational $t$, and therefore for all real t by the convexity of $\log \tilde Z_n(t)$. This ends the proof of Theorem \ref{LTT2.1.4}.

%%%%%%%%%%%%%%%%%%%%%%%%%%%%         section 5        %%%%%%%%%%%%%%%%%%%%%%%%%%%
\section {Large deviations for $ {\mathbb{E}_\xi \frac{Z_n(n\cdot)}{Z_n(\mathbb{R})}}$}\label{LTS5}
Using the lower bound in Therorem \ref{LTT2.1.5}  and the upper bound Theorem \ref{LTT2.1.1}, we have the following theorem.
\begin{thm}\label{LTT2.1.8}
  If  a.s. $\mathbb{P}_\xi(N\leq1)=0 $ and $\mathbb{E}_\xi N^{1+\delta}\leq K$ for some constants $\delta>0$
  and $K>0$, then a.s.,  for each measurable subset $A$ of $\mathbb{R}$,
\begin{eqnarray*}\label{LDP4}
   - \inf_{x\in A^o} \tilde\Lambda^* (x)-\mathbb{E}\log m_0
   &\leq& \liminf_{n\rightarrow \infty}
             \frac{1}{n}  \log \mathbb{E}_\xi\frac{Z_n (nA) }{Z_n(\mathbb{R})}\\
  & \leq &  \limsup_{n\rightarrow \infty}
            \frac{1}{n} \log  \mathbb{E}_\xi\frac{Z_n (nA) }{Z_n(\mathbb{R})}
   \leq - \inf_{x\in \bar A}  \Lambda^*(x)-\mathbb{E}\log m_0, \\
\end{eqnarray*}
where $A^o$ denotes the interior of $A$, and $\bar A$ its closure.
\end{thm}

Notice that $ \tilde\Lambda^* (x)=\Lambda^* (x)$ for $x\in (\Lambda'(t_-), \Lambda'(t_+))$. From Theorem \ref{LTT2.1.8} we obtain
\begin{co}\label{LTC2.1.9}
If  a.s. $\mathbb{P}_\xi(N\leq1)=0$ and $\mathbb{E}_\xi N^{1+\delta}\leq K$
 for some constants $\delta>0$ and $K>0$, then a.s.,
 $$ \lim_{n \rightarrow\infty} \frac {1}{n}  \log  \mathbb{E}_\xi\frac{Z_n [nx, \infty) }{Z_n(\mathbb{R})}
 = -\Lambda^*(x)-\mathbb{E}\log m_0 \mbox{ if } x \in (\Lambda'(0),
 \Lambda'(t_+)), $$
$$ \lim_{n \rightarrow\infty} \frac {1}{n}  \log  \mathbb{E}_\xi\frac{Z_n (-\infty, nx] }{Z_n(\mathbb{R})}
 = -\Lambda^*(x) -\mathbb{E}\log m_0 \mbox{ if } x \in (\Lambda'(t_-), \Lambda'(0)).$$
\end{co}

Theorem {\ref{LTT2.1.8}} is a combination of Lemmas \ref{LTL2.4.1} and \ref{LTL2.4.4} below.

\begin{lem}[Lower bound]\label{LTL2.4.1}
It is a.s. that for each measurable subset $A$ of
$\mathbb{R}$,
\begin{equation}\label{LTE2.4.1}
 \liminf_{n\rightarrow\infty}\frac{1}{n}\log \mathbb{E}_\xi\left(\frac{Z_n(nA)}{Z_n(\mathbb{R})}\right)
\geq-\inf_{x\in A^o}\tilde{\Lambda}^*(x)-E\log m_0.
\end{equation}
\end{lem}

\begin{proof}[Proof]
By Therorem \ref{LTT2.1.5}, a.s.
$$\liminf_{n\rightarrow\infty}\frac{1}{n}\log Z_n(nA)\geq -\inf_{x\in A^o}\tilde{\Lambda}^*(x),$$
which implies that for each $\varepsilon >0$, a.s.
$$
\lim_{n\rightarrow\infty}\mathbb{P}_\xi \left(\frac{1}{n}\log \frac{Z_n(nA)}{Z_n(\mathbb{R})}\geq-\tilde{\Lambda}^*(x)-\mathbb{E}\log m_0-\varepsilon \right)=1.
$$
Write $f(A)=-\inf_{x\in A^o}\tilde{\Lambda}^*(x)-\mathbb{E}\log m_0$.
Notice that
\begin{eqnarray*}
 \mathbb{E}_\xi\left(\frac{Z_n(nA)}{Z_n(\mathbb{R})} \right)
&\geq&\mathbb{E}_\xi\left(\frac{Z_n(nA)}{Z_n(\mathbb{R})} {1}_{\{\frac{Z_n(nA)}{Z_n(\mathbb{R})}\geq
\exp{(n(f(A)-\varepsilon))} \}} \right)\\
&\geq&\exp{\left(n(f(A)-\varepsilon)\right)}\mathbb{P}_\xi\left(\frac{1}{n}\log
\frac{Z_n(nA)}{Z_n(\mathbb{R})}\geq f(A)-\varepsilon \right).
\end{eqnarray*}
We have a.s.
$$\frac{1}{n}\log \mathbb{E}_\xi\left(\frac{Z_n(nA)}{Z_n(\mathbb{R})} \right)\geq f(A)-\varepsilon+\frac{1}{n}\log\mathbb{P}_\xi\left(\frac{1}{n}\log
\frac{Z_n(nA)}{Z_n(\mathbb{R})}\geq f(A)-\varepsilon \right).$$
Taking inferior limit and letting $\varepsilon\rightarrow 0$, we obtain (\ref{LTE2.4.1}).
\end{proof}

To obtain the upper bound, we need certain moment conditions.

\begin{lem}[\cite{huang}, Theorem 3.1]\label{LTL2.4.2}
If  a.s. $\mathbb{P}_\xi(N\leq1)=0 $ and $\mathbb{E}_\xi N^{1+\delta}\leq K$ for some constants $\delta>0$
  and $K>0$, then for each $s>0$, there exists a constants $C_s>0$ such that $\mathbb{E}_\xi W^{-s}\leq C_s$ a.s..
\end{lem}

\begin{lem}\label{LTL2.4.3}
If  a.s. $\mathbb{P}_\xi(N\leq1)=0 $ and $\mathbb{E}_\xi N^{1+\delta}\leq K$ for some constants $\delta>0$
  and $K>0$, then a.s.
\begin{equation}\label{LTE2.4.2}
\lim _{n\rightarrow\infty}\frac{1}{n}\log \mathbb{P}_\xi(Z_n(\mathbb{R})\leq
e^{(\mathbb{E}\log m_0-\varepsilon)n})=-\infty.
\end{equation}
\end{lem}

\begin{proof}[Proof]
Denote $W_n= Z_n(\mathbb{R})/P_n$.
Notice that $\forall s>0$, $\sup_n\mathbb{E}_\xi W_n^{-s}=\mathbb{E}_\xi W^{-s}$.  Lemma \ref{LTL2.4.2} shows that
$\mathbb{E}_\xi W^{-s}<\infty$ a.s.. By Markov's inequality, a.s.
\begin{eqnarray*}
&&\mathbb{P}_\xi\left(Z_n(\mathbb{R})\leq e^{(\mathbb{E}\log m_0-\varepsilon)n}\right)\\
&\leq&\mathbb{E}_\xi W_n^{-s}\exp{\left(s\left((\mathbb{E}\log
m_0-\varepsilon)n-\sum_{i=0}^{n-1}\log
m_i\right)\right)}\\
&\leq&\mathbb{E}_\xi W^{-s}\exp{\left(s\left((\mathbb{E}\log
m_0-\varepsilon)n-\sum_{i=0}^{n-1}\log
m_i\right)\right)}.\\
\end{eqnarray*}
Hence a.s.
$$\frac{1}{n}\log \mathbb{P}_\xi\left(Z_n(\mathbb{R})\leq e^{(\mathbb{E}\log m_0-\varepsilon)n}\right)
\leq\frac{1}{n}\log \mathbb{E}_\xi W^{-s}+s\left(\mathbb{E}\log
m_0-\varepsilon-\frac{1}{n}\sum_{i=0}^{n-1}\log m_i\right).$$ Taking
superior limit, we get a.s.
$$\limsup_{n\rightarrow\infty}\frac{1}{n}\log
\mathbb{P}_\xi\left(Z_n(\mathbb{R})\leq e^{(E\log m_0-\varepsilon)n}\right)
\leq-\varepsilon s.$$
Letting $s\rightarrow \infty$, we obtain (\ref{LTE2.4.2}).
\end{proof}

\begin{lem}[Upper bound]\label{LTL2.4.4}  If  a.s. $\mathbb{P}_\xi(N\leq1)=0 $ and $\mathbb{E}_\xi N^{1+\delta}\leq K$ for some constant $\delta>0$
  and $K>0$,
then it is a.s. that for each measurable subset $A$ of
$\mathbb{R}$,
\begin{equation}\label{LTE2.4.3}
 \limsup_{n\rightarrow\infty}\frac{1}{n}\log \mathbb{E}_\xi\left(\frac{Z_n(nA)}{Z_n(\mathbb{R})}\right)
\leq-\inf_{x\in \bar A}{\Lambda}^*(x)-E\log m_0.
\end{equation}
\end{lem}

\begin{proof}[Proof]
Notice that for each $\varepsilon>0$, a.s.
\begin{eqnarray*}
\mathbb{E}_\xi\left(\frac{Z_n(nA)}{Z_n(\mathbb{R})}\right)
&=&\mathbb{E}_\xi\left(\frac{Z_n(nA)}{Z_n(\mathbb{R})} {1}_{\{Z_n(\mathbb{R})>e^{(\mathbb{E}\log
m_0-\varepsilon)n}\}} \right)+\mathbb{E}_\xi\left(\frac{Z_n(nA)}{Z_n(\mathbb{R})}
 {1}_{\{Z_n(\mathbb{R})\leq e^{(\mathbb{E}\log
m_0-\varepsilon)n}\}} \right)\\
&\leq&e^{-(\mathbb{E}\log
m_0-\varepsilon)n}\mathbb{E}_\xi Z_n(nA)+\mathbb{P}_\xi\left(Z_n(\mathbb{R})\leq
e^{(\mathbb{E}\log m_0-\varepsilon)n}\right).
\end{eqnarray*}
Hence a.s.
$$\frac{1}{n}\log \mathbb{E}_\xi\left(\frac{Z_n(nA)}{Z_n(\mathbb{R})}\right)
\leq \frac{1}{n}\log \left(e^{-(\mathbb{E}\log
m_0-\varepsilon)n}\mathbb{E}_\xi Z_n(nA)+\mathbb{P}_\xi\left(Z_n(\mathbb{R})\leq
e^{(\mathbb{E}\log m_0-\varepsilon)n}\right)\right).$$
Taking superior limit in the above inequality, and using Theorem \ref{LTT2.1.1} and Lemma \ref{LTL2.4.3}, we obtain a.s.
\begin{eqnarray*}
&&\limsup_{n\rightarrow\infty}\frac{1}{n}\log
\mathbb{E}_\xi\left(\frac{Z_n(nA)}{Z_n(\mathbb{R})} \right)\\
&\leq&\max\left\{\limsup_{n\rightarrow\infty}\frac{1}{n}\log \mathbb{E}_\xi Z_n(nA)-\mathbb{E}\log m_0+\varepsilon,
\;\limsup_{n\rightarrow\infty}\frac{1}{n}\log{\mathbb{P}_\xi\left(Z_n(\mathbb{R})\leq
e^{(\mathbb{E}\log m_0-\varepsilon)n}\right)}\right\}\\
&=&\max\{-\inf_{x\in \bar{A}}\Lambda^*(x)-\mathbb{E}\log m_0+\varepsilon,
\;-\infty\}
=-\inf_{x\in \bar{A}}\Lambda^*(x)-\mathbb{E}\log m_0+\varepsilon.
\end{eqnarray*}
Then let $\varepsilon \rightarrow 0$.
\end{proof}

%%%%%%%%%%%%%%%%%%%%%%%%%%             section 6                  %%%%%%%%%%%%%%%%%%%%%%%%%%%%%%%%%
\section{Branching random walk in varying environment}\label{LTS6}
Kaplan and Asmussen (1976, \cite{ka}) showed that under certain moment conditions, the probability measures $\frac{Z_n(b_n\cdot+a_n)}{Z_n(\mathbb{R})}$ satisfy a central limit  and a local limit theorem for a branching random walk in deterministic environment for some sequence $(a_n, b_n)$.  Biggins (1990, \cite{b}) proved the same results under weaker moments conditions. We want to generalize these results to branching random walk with random environment in time. But instead of studying the case of random environment directly, we first introduce branching random walk with varying environment  in time and give some related results.

A branching random walk with a varying environment in time is modeled in a similar way as the branching random walk with a random environment in time.  Let $\{X_n\}$ be a sequence of point processes on $\mathbb{R}$.  The distribution of $X_n$ is denoted by $\eta_n$. At time $0$, there is an initial particle $\emptyset$ of generation $0$ located at $S_{\emptyset}=0$;
at time $1$, it is replaced by $N=N({\emptyset})$ particles of generation $1$, located at $L_i=L_i(\emptyset)$, $1\leq i\leq N$, where the point process
$X_{\emptyset}=(N, L_1, L_{2},\cdots)$ is an independent copy of $X_0$ . In general, each particle $u=u_1\cdots u_n$ of generation $n$ located at $S_u$ is replaced at
time {n+1} by $N(u)$ new particles $ui$ of generation $n+1$, located at
$$S_{ui}=S_u+L_i(u)\qquad(1\leq i\leq N(u)),$$
where the point process formulated by the number of offspring and there displacements, $\{X(u)=(N(u),L_1(u),L_2(u),\cdots)\}$, is an independent copy of $X_n$.
All particles behave independently, namely, the point processes  $\{X(u)\}$ are independent of each other. In particular,  $\{X(u): u\in\mathbb{T}_n\}$ are independent  of each other and have a common distribution  $\eta_n$.

Let
$Z_n=\sum_{u\in\mathbb{T}_n}\delta_{S_u}$
be the counting measure of particles of generation $n$. As the case of random environment, we introduce the following notations:
\begin{equation}\label{LTE6.1}
N_n=X_n(\mathbb{R}),\qquad m_n=\mathbb{ \mathbb{E}}N_n,
 \qquad P_0=1\quad \text{and}\quad P_n=\mathbb{E}Z_n(\mathbb{R})=\prod_{i=0}^{n-1}m_i.
\end{equation}
Assume that
\begin{equation}\label{LTE1.1.1}
0<m_n<\infty,\quad\liminf_{n\rightarrow\infty}\frac{1}{n}\log P_{n}>0\quad \text{and}\quad \liminf_{n\rightarrow\infty}\frac{1}{n}\log
m_n=0.
\end{equation}
Thus for some $c>1$, there exists an integer $n_0$ depending on $c$
such that
\begin{equation}\label{LTE1.1.2}
P_{n}>c^n \qquad\text{for all}\;n>n_0
\end{equation}

Denote $\Gamma$ the probability space under which the process is defined. Let  $\mathcal {F}_0=\{\emptyset, \Gamma\}$ and  $\mathcal {F}_n=\sigma((N(u), L_1(u), L_2(u),\cdots): |u|<n)$ for $n\geq 1$ be the $\sigma$-field containing all the
information concerning the first $n$ generations, then the sequence
$\{Z_n(\mathbb{R})/P_{n}\}$ forms a non-negative martingale with respect
to the filtration $\mathcal {F}_n$ and converges a.s. to a random variable $W$.

Let $\nu_n$ be the intensity measure of the point process $\frac{X_n}{m_n}$ in the sense that for a subset $A$ of $\mathbb{R}$,
$$\nu_n(A)=\frac{\mathbb{E}X_n(A)}{m_n},$$ and let
$\phi_n$  be the corresponding characteristic function, i.e.
\begin{equation}
\phi_n(t)=\int e^{itx}\nu_n(dx)=\frac{1}{m_n} \mathbb{E}\int e^{itx}X_n(dx).
\end{equation}
The characteristic function of $\frac{Z_{n}}{P_{n}}$ is defined
as
\begin{equation}
\Psi_{n}(t)=\frac{1}{P_{n}}\int e^{itx}Z_{n}(dx)=\frac{1}{P_{n}}\sum_{u\in\mathbb{T}_n}e^{itS_u}.
\end{equation}
It is not difficult to see that $\phi_i $ and $\Psi_{n}$ have the following relation:
\begin{equation}
\mathbb{E}\Psi_{n}(t)=\prod_{i=0}^{n-1}\phi_i(t).
\end{equation}

Furthermore, denote
\begin{equation}
\upsilon_n(\varepsilon)=\sum_{i=0}^{n-1}\int|x|^{\varepsilon}\nu_i(dx).
\end{equation}

\noindent\textbf{Condition (A)}. \emph{There is a non-degenerate probability distribution
$L(x)$ and constants $ \{a_n,b_n\}$ with $b_n\rightarrow\infty$
such that}
\begin{equation*}
e^{-ita_n}\prod_{i=0}^{n-1}\phi_i(t/b_n)\rightarrow g(t)=\int e^{itx}L(dx).
\end{equation*}

Similar conditions were posed by Klebaner (1982, \cite{k})  and Biggins (1990, \cite{b}). If additionally $b_{n+1}/b_n\rightarrow1$,
then the limit distribution would be in What Feller (1971, \cite{fel}) calls the class $L$, also known as the self-decomposable distributions.
\\*

Denote $G_n(x)=\nu_0\ast\cdots\ast\nu_{n-1}(x)$, we introduce
another condition:
 \\*

\noindent\textbf{Condition (B)}. \emph{There exist constants $ \{a_n,b_n\}$ with
$b_n\rightarrow\infty$ such that $G_n(b_nx+a_n)$ converges to
a non-degenerate probability distribution $L(x)$ }.

It is clear that if (B) holds with $ \{a_n,b_n\}$, then (A) holds with
$a_n'=\frac{a_n+o(b_n)}{b_n} $ and $b_n'=b_n$.  Let $\mu_n=\int
x\nu_n(dx) $ and $\sigma_n^2=\int|x-\mu_n|^2\nu_n(dx)$.   Take $a_n=\sum_{i=0}^{n-1}\mu_i$ and
$b_n=(\sum_{i=0}^{n-1}\sigma_i^2)^{1/2}$, if moreover $b_n$ satisfying
$b_{n+1}/b_n\rightarrow 1$, then $G_n(b_nx+a_n)\rightarrow L(x)$. In
particular, if $\{\nu_n\}$ satisfies Lindeberg or Liapounoff
conditions, then the limiting distribution $L$ is standard normal,
i.e.
$L(x)=\Phi(x)=\frac{1}{\sqrt{2\pi}}\int_{-\infty}^xe^{-t^2/2}dt$.
\\*

We have the following result:

\begin{thm}\label{LTT1.1.1}
For a branching random walk in a varying environment
satisfying (\ref{LTE1.1.1}), assume that for some $\delta>0$,
\begin{equation}\label{LTE1.1.3}
\sum_n\frac{1}{m_nn(\log n)^{1+\delta}}\mathbb{E}N_n\log^+N_n
(\log^+\log^+N_n)^{1+\delta}<\infty,
\end{equation}
for some $\varepsilon>0$ and $\gamma_1<\infty$,
\begin{equation}\label{LTE1.1.41}
\upsilon_n(\varepsilon)=o(n^{\gamma_1}),
\end{equation}
and for some $\gamma_2>0$,
\begin{equation}\label{LTE1.1.4}
b_n^{-1}=o(n^{-\gamma_2}),
\end{equation}
then
\begin{equation}\label{LTE1.1.5}
\Psi_{n}(t/b_n)-W\prod_{i=0}^{n-1}\phi_i(t/b_n)\rightarrow0\qquad a.s..
\end{equation}
If in addition (A) holds, then
\begin{equation}\label{LTE1.1.6}
 e^{-ita_n}\Psi_{n}(t/b_n)\rightarrow g(t)W\qquad a.s.,
\end{equation}
 and for $x$ a continuity point of $L$,
\begin{equation}\label{LTE1.1.7}
P_n^{-1}Z_n(-\infty,b_n(x+a_n)]\rightarrow L(x)W\qquad
a.s..
\end{equation}
The null set can be taken to be independent of $t$ in (\ref{LTE1.1.6}) and $x$ in (\ref{LTE1.1.7})
respectively, and (\ref{LTE1.1.6}) holds uniformly for $t$ in compact sets.
\end{thm}

\noindent \emph{Remark.} The above conclusions were obtained by Biggins (1990, \cite{b}, Theorem 1 and 2)
under similar hypothesis with (\ref{LTE1.1.3}) replaced by a condition $\int x\log x F(dx)<\infty$, where $F(x):=\sum_{k=0}^{[x]}\sup_nP(N_n=k)$.
In homogeneous case, $F$ is simply the distribution function which determines the offspring's number, but in general, $F$ has not such a concrete expression as (\ref{LTE1.1.3}).
\\*

The following theorem is a local limit theorem. We use the notation $a_n\sim b_n$ to
 signify that $a_n/b_n\rightarrow1$ as $n\rightarrow\infty$.

 \begin{thm}\label{LTT1.2}
For a branching random walk in a varying environment satisfying (\ref{LTE1.1.1}),
assume that (A) holds with $b_n\sim \theta n^\gamma $ for some constants
$0<\gamma\leq\frac{1}{2}$ and $\theta>0$,  $g$ is integrable
and for some $\iota>0$,
\begin{equation}\label{LTET1.1.21}
\sup_i\sup_{ |u|\geq\iota}|\phi_i(t)|=:c_\iota<1.
\end{equation}
If (\ref{LTE1.1.41}) holds and
\begin{equation}\label{LTET1.1.22}
\sum_n\frac{1}{m_nn(\log n)^{1+\delta}} \mathbb{E}N_n(\log^+N_n)^{1+\beta}
<\infty
\end{equation}
for some $\delta>0$ and $\beta>\gamma$, then
\begin{equation}
\sup_{x\in \mathbb{R}}\left|b_nP_n^{-1}Z_n(x,x+h)-Whp_L(x/b_n-a_n)\right|\rightarrow0\qquad
a.s.,
\end{equation}
where $p_L(x)$ denotes the density function of $L$.
\end{thm}

%%%%%%%%%%%%%%%%%%%%%%%%%%%%           section 7           %%%%%%%%%%%%%%%%%%%%%%%%%%%%%%%%%%%
\section{Proof of Theorem \ref{LTT1.1.1}}\label{LTS7}
To prove Theorem \ref{LTT1.1.1}, we only need to show (\ref{LTE1.1.5}), for it is obvious that (\ref{LTE1.1.6}) is directly from (\ref{LTE1.1.5}),  and (\ref{LTE1.1.7}) is from (\ref{LTE1.1.6}) by applying the continuity theorem. The rest assertions are according to Biggins (1990, \cite{b},  Theorem 2) . We remark here that our proof is inspired by
Biggins (1990, \cite{b}) and Klebaner (1982, \cite{k}).

We will use a truncation method. Let $\kappa>0$ be a constant. Let
$\tilde{X}_{n,\kappa}$ be equal to $X_n$ on $\{N_n(\log
N_n)^{\kappa}\leq P_{n+1}\}$ and be empty otherwise; the rest of
the notations is extended similarly. Let
$I_n(x)= {1}_{\{x(\log x)^{\kappa}\leq P_{n+1}\}}$ and
$I_n^c=1-I_n$, so
$$\tilde{m}_{n,\kappa}= \mathbb{E}N_nI_n(N_n),$$
and$$\tilde{\phi}_{n,\kappa}(t)=\int
e^{itx}\tilde{\nu}_{n,\kappa}(dx)=\frac{1}{\tilde{m}_{n,\kappa}} \mathbb{E}\int
e^{itx}X_n(dx)I_n(N_n).$$

The proof of Theorem \ref{LTT1.1.1} is composed of several  lemmas.

\begin{lem}\label{LTL1.2.1}
Let $\beta\geq0$.
If $\sum_n\frac{1}{m_nn(\log n)^{1+\delta}}\mathbb{E}
N_n(\log^+N_n)^{1+\beta} (\log^+\log^+N_n)^{1+\delta}<\infty$ holds
for some $\delta>0$, then for all $\kappa$, $\sum_n
n^{\beta}(1-\tilde{m}_{n,\kappa}/m_n)<\infty$.
\end{lem}

\begin{proof}[Proof]
We can calculate
\begin{eqnarray*}
 \sum_nn^{\beta}(1-\frac{\tilde{m}_{n,\kappa}}{m_n})&=&\sum_n\frac{n^{\beta}}{m_n}(m_n-\tilde{m}_{n,\kappa})\\
&=&\sum_n\frac{n^{\beta}}{m_n} \mathbb{E}N_nI^c_n(N_n)\\
&=&\sum_n\frac{n^{\beta}}{m_n} \mathbb{E}N_nI^c_n(N_n) {1}_{\{N_n>a\}}+\sum_n\frac{n^{\beta}}{m_n} \mathbb{E}N_nI^c_n(N_n) {1}_{\{N_n\leq
a\}},
\end{eqnarray*}
where $a$ is a constant. Since $P_n\rightarrow\infty$, the convergence of the second
series above is obvious. It suffices to show that of the first series for suitable $a$. Take $f(x)=(\log x)^{1+\beta}(\log\log x)^{(1+\delta)}$. $f(x)$ is
increasing and positive on $(a,+\infty)$. Noticing (\ref{LTE1.1.2}), we have
for $n$ large enough,
\begin{eqnarray*}
&&\frac{n^{\beta}}{m_n} \mathbb{E}N_nI^c_n(N_n) {1}_{\{N_n>a\}}\\
&\leq&\frac{n^{\beta}}{m_n}\mathbb{E}\frac{f(N_n(\log N_n)^{\kappa})}{f(
P_{n+1})} {1}_{\{N_n>a\}}\\
&\leq&\frac{C}{m_nn(\log n)^{1+\delta}} \mathbb{E}N_n(\log
N_n)^{1+\beta}(\log\log
N_n)^{1+\delta} {1}_{\{N_n>a\}}\\
&\leq&\frac{C}{m_nn(\log n)^{1+\delta}} \mathbb{E}N_n(\log^+
N_n)^{1+\beta}(\log^+\log^+ N_n)^{1+\delta},
\end{eqnarray*}
where $C$ is a constant, and like $a$, in general, it does not
necessarily stand for the same constant throughout.
The convergence of the series $\sum_n\frac{1}{m_nn(\log
n)^{1+\delta}} \mathbb{E}N_n(\log^+N_n)^{1+\beta}
(\log^+\log^+N_n)^{1+\delta}$ implies that of the series
$\sum_n\frac{n^{\beta}}{m_n} \mathbb{E}N_nI^c_n(N_n) {1}_{\{N_n>a\}}$.
\end{proof}

\begin{lem}[\cite{b}, Lemma 3 (ii)]\label{LTL1.2.2}
If $\sum_n(1-\tilde{m}_{n,\kappa}/m_n)<\infty$, then
\begin{equation}\label{LTE1.2.1}
\prod_{i=0}^{n-1}\tilde{\phi}_{i,\kappa}(t/b_n)-\prod_{i=0}^{n-1}
\phi_i(t/b_n)\rightarrow0,\qquad as\; n\rightarrow\infty.
\end{equation}
\end{lem}

The formula (\ref{LTE1.2.1}) shows that we can prove  (\ref{LTE1.1.5}) with $\tilde{\phi}_{i,\kappa}$
in place of $\phi_i$. For simplicity, let
$$\zeta_n(t)=\tilde{\phi}_{n,\kappa}(t)\qquad\text{and}\qquad \omega_n=\frac{\tilde{m}_{n,\kappa}}{m_n},$$
where the value of $\kappa$ will be fixed to be suitably large later.

Let $\Psi_{u}^{(1)}(t):=m_{n}^{-1}\int e^{itx}X(u)(dx)$ if $u\in\mathbb{T}_n$. Then
\begin{eqnarray*}
&&\Psi_{n+1}(t)-\omega_n\zeta_n(t)\Psi_n(t)\\
&=&\frac{1}{P_n}\sum_{u\in\mathbb{T}_n}e^{itS_u}\Psi_u^{(1)}(t)I^c_n(N(u))
+\frac{1}{P_n}\sum_{u\in\mathbb{T}_n}e^{itS_u}\left(\Psi_u^{(1)}(t)I_n(N(u))-\omega_n\zeta_n(t)\right)\\
&=&:A_n(t)+B_n(t).
\end{eqnarray*}
By iteration, we obtain
\begin{equation}\label{LTE1.2.2}
\Psi_n\left(\frac{t}{b_n}\right)-\Psi_k\left(\frac{t}{b_n}\right)\prod_{i=k}^{n-1}\omega_i\zeta_i\left(\frac{t}{b_n}\right)
=\sum_{i=k}^{n-1}\left(A_i\left(\frac{t}{b_n}\right)+B_i\left(\frac{t}{b_n}\right)\right)\prod_{j=i+1}^{n-1}\omega_j\zeta_j\left(\frac{t}{b_n}\right).
\end{equation}
Thus
\begin{eqnarray}\label{LTE1.2.3}
&&\Psi_n({t}/{b_n})-W\prod_{i=0}^{n-1}\zeta_i(t/b_n)\nonumber\\
&=&\sum_{i=k}^{n-1}A_i(t/b_n)\prod_{j=i+1}^{n-1}\omega_j\zeta_j(t/b_n)
+\sum_{i=k}^{n-1}B_i(t/b_n)\prod_{j=i+1}^{n-1}\omega_j\zeta_j(t/b_n)\nonumber\\
&&+\left(\Psi_k(t/b_n)\prod_{i=k}^{n-1}\omega_i\zeta_i(t/b_n)-W\prod_{i=0}^{n-1}\zeta_i(t/b_n)\right).
\end{eqnarray}

Let $\alpha>1$. Take $k=J(n)=j$ if $j^{\alpha}\leq
n<(j+1)^{\alpha}$, so that $k^{\alpha}\sim n$, which means $k$ goes
to infinity more slowly than $n$. For this $k$, we will show that
each term in the right side of (\ref{LTE1.2.3}) is negligible.

\begin{lem}\label{LTL1.2.3}
If $\sum_n(1-\tilde{m}_{n,\kappa}/m_n)<\infty$, then
\begin{equation}\label{LTE1.2.4}
\sum_{i=k}^{n-1}A_i(t/b_n)\prod_{j=i+1}^{n-1}\omega_j\zeta_j(t/b_n)\rightarrow0\qquad
a.s.,\qquad as\;n\rightarrow\infty.
\end{equation}
\end{lem}

\begin{proof}[Proof]
Notice that
\begin{equation}\label{LTE1.2.5}
\left|\sum_{i=k}^{n-1}A_i\prod_{j=i+1}^{n-1}\omega_j\zeta_j\right|\leq\sum_{i=k}^{n-1}|A_i|\leq\sum_{i=k}^{n-1}\frac{1}{P_im_i}\sum_{|u|=i}N(u)I^c_i(N(u)).
\end{equation}
Since
\begin{eqnarray*}
 \mathbb{E}\left(\sum_{i=0}^{\infty}\frac{1}{P_im_i}\sum_{|u|=i}N(u)I^c_i(N(u))\right)
=\sum_{i=0}^{\infty}\frac{1}{m_i}\mathbb{E} N_iI^c_i(N_{i})
=\sum_i(1-\frac{\tilde{m}_{i,\kappa}}{m_i})<\infty.
\end{eqnarray*}
we get
$$\sum_{i=0}^{\infty}\frac{1}{P_im_i}\sum_{|u|=i}N(u)I^c_i(N(u))<\infty.$$
which implies (\ref{LTE1.2.4}), combined with (\ref{LTE1.2.5}).
\end{proof}

\begin{lem}\label{LTL1.2.4}
If for some $\delta_1>0$ ,
\begin{equation}\label{LTE1.2.6}
\sum_n\frac{1}{m_nn^{1+\delta_1}} \mathbb{E}N_n\log^+N_n<\infty,
\end{equation}
then
\begin{equation}\label{LTE1.2.7}
\sum_{i=k}^{n-1}B_i(t/b_n)\prod_{j=i+1}^{n-1}\omega_j\zeta_j(t/b_n)\rightarrow0\qquad
a.s.,\qquad as\;n\rightarrow\infty.
\end{equation}
\end{lem}
\noindent\emph{Remark}. Obviously (\ref{LTE1.1.3}) implies (\ref{LTE1.2.6}).

\begin{proof}[Proof]
Let
$$C_n=\sum_{i=k}^{n-1}B_i(t/b_n)\prod_{j=i+1}^{n-1}\omega_j\zeta_j(t/b_n).$$
We want to show that $\sum_{n=1}^{\infty}\mathbb{E}|C_n|^2<\infty$, which
implies (\ref{LTE1.2.6}). Since $E(B_i|\mathcal {F}_i)=0$, we have
$$\mathbb{E}|C_n|^2=var(C_n)=var\left(\sum_{i=k}^{n-1}B_i\prod_{j=i+1}^{n-1}\omega_j\zeta_j\right)\leq\sum_{i=k}^{n-1}var(B_i),$$
where the notation $var$ denotes variance.
Moreover,
\begin{eqnarray*}
var(B_i)&=&\mathbb{E}(var(B_i|\mathcal {F}_i))
\leq\frac{1}{P_i}var(\Psi_i^{(1)}I_i(N_i))
\leq\frac{1}{P_im_i^2}\mathbb{E}N_i^2I_i(N_i),
\end{eqnarray*}
where $\Psi_{n}^{(1)}(t):=m_{n}^{-1}\int e^{itx}X_{n}(dx)$.
We denote $J^{-1}$ be the inverse mapping of $J$,
$J^{-1}(j)=\{n:J(n)=j\}$ and $|J^{-1}(j)|$ be the number of the
elements in $J^{-1}(j)$. It is not difficult to see that
$|J^{-1}(j)|=O(j^{\alpha-1})$ and
$\sum_{j=1}^i|J^{-1}(j)|=O(i^\alpha)$. Hence,
\begin{eqnarray*}
\sum_{n=1}^\infty
\mathbb{E}|C_n|^2&\leq&\sum_{n=1}^\infty\sum_{i=k}^{n-1}\frac{1}{P_im_i^2}\mathbb{E}N_i^2I_i(N_i)\\
&=&\sum_{j=1}^\infty\sum_{n\in
J^{-1}(j)}\sum_{i=j}^{n-1}\frac{1}{P_im_i^2}\mathbb{E}N_i^2I_i(N_i)\\
&\leq&\sum_{j=1}^\infty|J^{-1}(j)|\sum_{i=j}^{\infty}\frac{1}{P_im_i^2}\mathbb{E}N_i^2I_i(N_i)\\
&=&\sum_{i=1}^\infty\sum_{j=1}^{i}|J^{-1}(j)|\frac{1}{P_im_i^2}\mathbb{E}N_i^2I_i(N_i)\\
&\leq&C\sum_{i=1}^\infty\frac{i^{\alpha}}{P_im_i^2}\mathbb{E}N_i^2I_i(N_i)\\
&=&C\sum_{i=1}^\infty\frac{i^{\alpha}}{P_im_i^2}\mathbb{E}N_i^2I_i(N_i) {1}_{\{N_i>a\}}
+C\sum_{i=1}^\infty\frac{i^{\alpha}}{P_im_i^2}\mathbb{E}N_i^2I_i(N_i) {1}_{\{N_i\leq
a\}}.
\end{eqnarray*}
The second series above converges, since
$\sum_i\frac{i^\alpha}{P_im_i^2}<\infty$. For the first series
above, take $f(x)=x(\log x)^{-(\alpha+1+\delta_1)}$. $f(x)$ is
increasing and positive on $(a,+\infty)$. We have for $i$ large
enough,
\begin{eqnarray*}
&&\frac{i^{\alpha}}{P_im_i^2}\mathbb{E}N_i^2I_i(N_i) {1}_{\{N_i>a\}}\\
&\leq&\frac{i^{\alpha}}{P_im_i^2}\mathbb{E}N_i^2\frac{f(P_{i+1})}{f(\{N_i(\log
N_i)^{\kappa})} {1}_{\{N_i>a\}}\\
&\leq&\frac{C}{m_ii^{1+\delta_1}}\mathbb{E}N_i(\log
N_i)^{\alpha+1+\delta_1-\kappa} {1}_{\{N_i>a\}}\\
&\leq&\frac{C}{m_ii^{1+\delta_1}}\mathbb{E}N_i\log^+ N_i,
\end{eqnarray*}
if we take $\kappa\geq\alpha+\delta_1$. Then by (\ref{LTE1.2.6}), it follows the
convergence of the series
$\sum_i\frac{i^{\alpha}}{P_im_i^2}\mathbb{E}N_i^2I_i(N_i) {1}_{\{N_i>a\}}$.
\end{proof}

\begin{lem}\label{LTL1.2.5}
If $\sum_n(1-\tilde{m}_{n,\kappa}/m_n)<\infty$ and (\ref{LTE1.1.41}), (\ref{LTE1.1.4}) hold,
then
\begin{equation}\label{LTE1.2.8}
\Psi_k(t/b_n)\prod_{i=k}^{n-1}\omega_i\zeta_i(t/b_n)-W\prod_{i=0}^{n-1}\zeta_i(t/b_n)\rightarrow0\qquad
a.s.,\qquad as\;n\rightarrow\infty.
\end{equation}
\end{lem}

\begin{proof}[Proof]
$\sum_n(1-\tilde{m}_{n,\kappa}/m_n)<\infty$ implies
that $\sum_{i=k}^{n-1}\omega_i\rightarrow1$, so the factor
$\prod_{i=k}^{n-1}\omega_i$ in (2.8) can be ignored. Notice that
$$\Psi_k\prod_{i=k}^{n-1}\zeta_i-W\prod_{i=0}^{n-1}\zeta_i
=\left(\Psi_k-\frac{Z_k(\mathbb{R})}{P_k}\right)\prod_{i=k}^{n}\zeta_i +
\left(\frac{Z_k(\mathbb{R})}{P_k}-W\right)\prod_{i=k}^{n-1}\zeta_i
+W\left(\prod_{i=k}^{n-1}\zeta_i-\prod_{i=0}^{n-1}\zeta_i\right).
$$
It suffices to
prove that
\begin{equation}\label{LTE1.2.9}
\Psi_k(t/b_n)-\frac{Z_k(\mathbb{R})}{P_k}\rightarrow0\qquad
a.s.,\qquad as\;k\rightarrow\infty.
\end{equation}
and
\begin{equation}\label{LTE1.2.10}
\prod_{i=k}^{n-1}\zeta_i(t/b_n)-\prod_{i=0}^{n-1}\zeta_i(t/b_n)\rightarrow0,
\qquad as\;k\rightarrow\infty.
\end{equation}
Since $|e^{itx}-1|\leq C|tx|^\varepsilon$, we have
\begin{eqnarray*}
\left|\Psi_k\left({t}/{b_n}\right)-\frac{Z_n(\mathbb{R})}{P_k}\right|&\leq&\frac{1}{P_k}\int\left|e^{tb^{-1}_nx}-1\right|Z_k(dx)\\
&\leq&C|u|^\varepsilon
b_n^{-\varepsilon}\frac{1}{P_k}\int|x|^\varepsilon Z_k(dx).
\end{eqnarray*}
Assume that $0<\varepsilon\leq1$ (the proof for the case of
$\varepsilon>1$ is similar). Taking expectation in the above inequality, we obtain
\begin{eqnarray*}
\mathbb{E}\left|\Psi_k(t/b_n)-\frac{Z_k(\mathbb{R})}{ P_k}\right|&\leq&C|t|^\varepsilon
\sup\{b_n^{-\varepsilon}:k^\alpha\leq n<(k+1)^\alpha\}\int|x|^\varepsilon\nu_0\ast\cdots\ast\nu_{k-1}(dx)\\
&\leq&C|t|^\varepsilon
\sup\{b_n^{-\varepsilon}:k^\alpha\leq n<(k+1)^\alpha\}\sum^{k-1}_{i=0}\int|x|^\varepsilon\nu_i(dx)\\
&=&C|t|^\varepsilon \sup\{b_n^{-\varepsilon}:k^\alpha\leq
n<(k+1)^\alpha\}\upsilon_k(\varepsilon) =|u|^\varepsilon
o(k^{\gamma_1-\alpha\varepsilon\gamma_2}).
\end{eqnarray*}
Hence (\ref{LTE1.2.9}) holds if we take $\alpha$ large.
By Lemma \ref{LTL1.2.2}, we can prove (\ref{LTE1.2.10}) with $\phi_i$ in place of $\zeta_i$,
which holds directly by noticing that
\begin{eqnarray*}
\left|\prod_{i=k}^{n-1}\phi_i (t/b_n)-\prod_{i=0}^{n-1}\phi_i (t/b_n)\right|
&=&\left|\prod_{i=k}^{n-1}\phi_i (t/b_n)\left(1-\prod_{i=0}^{k-1}\phi_i (t/b_n)\right)\right|
\leq\left|1-\prod_{i=0}^{k-1}\phi_i(t/b_n)\right|\\
&=&\left|\mathbb{E}\left(\Psi_k(t/b_n)-\frac{Z_k(\mathbb{R})}{ P_k}\right)\right| \leq
\mathbb{E}\left|\Psi_k(t/b_n)-\frac{Z_k(\mathbb{R})}{ P_k}\right|.
\end{eqnarray*}
Thus (\ref{LTE1.2.8}) holds.
\end{proof}

%%%%%%%%%%%%%%%%%%%%%%%           section 8      %%%%%%%%%%%%%%%%%%%%%%%%%%%%%%%%
\section{Proof of Theorem \ref{LTT1.2}}\label{LTS8}
 We will go along the proof by following the lines in \cite{b}. Let
$$K(x)=\frac{1}{2\pi}\left(\frac{\sin\frac{1}{2}x}{\frac{1}{2}x}\right)^2\quad K_a(x)=\frac{1}{a}K(\frac{x}{a})\;(a>0).$$
Then $$\int_{\mathbb{R}}K(x)dx=1 \qquad\text{and}\qquad
\int_{\mathbb{R}}K_a(x)dx=1.$$ The characteristic function of $K_a$
is denoted by $k_a$, which vanishes outside
$(-\frac{1}{a},\frac{1}{a})$, so that the characteristic function of
$\frac{Z_n}{P_n}*K_a$ is integrable and so
$\frac{Z_n}{P_n}*K_a$ has a density function $D_a^{(n)}$. We
will get our result through the asymptotic property of $D_a^{(n)}$.

\begin{lem}[see \cite{c}]\label{LTL1.3.1}
If $f(t)$ is a characteristic function such that $|f(t)|\leq\kappa$ as soon as $b\leq|u|\leq2b$, then we have for $|u|<b$,
$$|f(t)|\leq1-(1-\kappa^2)\frac{t^2}{8b^2}.$$
\end{lem}

\begin{lem}\label{LTL1.3.2}
Under the conditions of Theorem \ref{LTT1.2},
\begin{equation}\label{den}
\sup_{x\in\mathbb{R}}|b_nD_a^{(n)}(b_n(x+a_n))-Wp_L(x)|\rightarrow0\;a.s.,\qquad as\;\;
n\rightarrow\infty.
\end{equation}
\end{lem}

\begin{proof}[Proof]
Let $A$ be a positive constant. By the Fourier
inversion theorem,
$$
2\pi\left|b_nD_a^{(n)}(b_n(x+a_n))-Wp_L(x)\right|
=\left|\int\left(\Psi_n(\frac{t}{b_n})k_a(\frac{t}{b_n})e^{-ita_n}-W
g(t)\right)dt\right|.
$$
Split the integral of the right side into $|t|<A$ and $|t|\geq A$.
Using Theorem \ref{LTT1.1.1} and noticing that $\lim_nk_a(t/b_n)=1$, we have
\begin{eqnarray*}
&&\left|\int_{|t|<A}\left(\Psi_n(\frac{t}{b_n})k_a(\frac{t}{b_n})e^{-iua_n}-W g(t)\right)dt\right|\\
&=& \int_{|t|<A}\left|\left(\Psi_n(\frac{t}{b_n})e^{-ita_n}-W
g(t)\right)k_a(\frac{t}{b_n})\right|dt
+ \int_{|t|<A}\left|(1-k_a(\frac{t}{b_n}))W g(t)\right|dt\\
&\leq&2A\sup_{|t|\leq A}\left|\Psi_n(\frac{t}{b_n})e^{-ita_n}-W
g(t)\right|+W\int_{|t|<A}\left|1-k_a(\frac{t}{b_n})\right|dt \rightarrow0\; a.s.,\quad
as \;n\rightarrow\infty.
\end{eqnarray*}
For $A$ large, the integral of $g(t)$ over $|t|\geq A$ is small. So
to show (\ref{den}), it remains to consider
$$\left|\int_{|t|\geq A} \Psi_n(\frac{t}{b_n})k_a(\frac{t}{b_n})dt\right|=\left|\int_Ub_n\Psi_n(t)k_a(t)dt\right|,$$
where $U=\{t:\frac{A}{b_n}\leq|t|\leq\frac{1}{a}\}$. By the
decomposition (\ref{LTE1.2.2}),
$$b_n \Psi_nk_a=b_n \Psi_k\prod_{i=k}^{n-1}\omega_i\zeta_ik_a+
b_n\sum^{n-1}_{i=k}A_i\prod_{j=i+1}^{n-1}\omega_j\zeta_jk_a+b_n\sum^{n-1}_{i=k}B_i\prod_{j=i+1}^{n-1}\omega_j\zeta_jk_a.$$
Take $k=J(n)$ the same as the proof of Theorem \ref{LTT1.1.1},we need to show
that
\begin{equation}\label{LTE1.3.2}
\left|b_n\sum_{i=k}^{n-1}\int_UA_i\prod_{j=i+1}^{n-1}\omega_j\zeta_jk_adt\right|\rightarrow0\quad a.s.,\qquad
as\;n\rightarrow\infty.
\end{equation}
and the similar result with $B_i$ in place of $A_i$.

Firstly, for $n$ large enough,
\begin{eqnarray*}
&&\left|b_n\sum_{i=k}^{n-1}\int_UA_i\prod_{j=i+1}^{n-1}\omega_j\zeta_jk_adt\right|\leq b_n\sum_{i=k}^{n-1}\int_U|A_i|dt\\
&\leq&C k^{\alpha\gamma}\sum_{i=k}^{n-1} \frac{1}{P_im_i}\sum_{|u|=i}N(u)I_i^c(N(u))\\
&\leq&C\sum_{i=k}^{n-1}\frac{i^{\alpha\gamma}}{P_im_i}\sum_{|u|=i}N(u)I_i^c(N(u)).
\end{eqnarray*}
Like the proof of Lemma \ref{LTL1.2.3},  we obtain
$$\mathbb{E}\left(\sum_{i=0}^{\infty} \frac{i^{\alpha\gamma}}{P_im_i}\sum_{|u|=i}N(u)I_i^c(N(u))\right)
=\sum_{i=0}^{\infty}i^{\alpha\gamma}\left(1-\frac{\tilde{m}_{i,k}}{m_i}\right)<\infty$$
from Lemma  \ref{LTL1.2.1}, if we take $\alpha$ sufficiently near $1$ such that
$\alpha\gamma<\beta$. Hence (\ref{LTE1.3.2}) is proved.

Secondly, to prove (\ref{LTE1.3.2}) with $B_i$ in place of $A_i$, like the
proof of Lemma \ref{LTL1.2.4}, we set
$$C_n=b_n\sum_{i=k}^{n-1}\int_UB_i\prod_{j=i+1}^{n-1}\omega_j\zeta_jk_adt.$$
Since $\mathbb{E}(B_i|\mathcal {F}_i)=0$,  for $n$ large enough,
\begin{eqnarray*}
\mathbb{E}|C_n|^2&=&var\left(b_n\sum_{i=k}^{n-1}\int_UB_i\prod_{j=i+1}^{n-1}\omega_j\zeta_jk_adt\right)\\
&=&b_n^2\sum_{i=k}^{n-1}var\left(\int_UB_i\prod_{j=i+1}^{n-1}\omega_j\zeta_jk_adt\right)\\
&=&b_n^2\sum_{i=k}^{n-1}\mathbb{E} \left|\int_UB_i\prod_{j=i+1}^{n-1}\omega_j\zeta_jk_adt\right |^2\\
&\leq&b_n^2\sum_{i=k}^{n-1}\mathbb{E}\left(\int_Udt\right)\left(\int_U|B_i\prod_{j=i+1}^{n-1}\omega_j\zeta_jk_a|^2dt\right)\\
&\leq&\frac{2}{a}b_n^2\sum_{i=k}^{n-1}\int_U\mathbb{E}|B_i|^2dt\\
&=&\frac{2}{a}b_n^2\sum_{i=k}^{n-1}\int_Uvar|B_i|^2dt\\
&\leq&C\sum_{i=k}^{n-1}\frac{i^{2\alpha\gamma}}{P_im_i^2}\mathbb{E} N_i^2I_i(N_{i}).
\end{eqnarray*}
Following the last part of the proof of Lemma \ref{LTL1.2.4}, we obtain that
$\sum_{n=1}^{\infty}\mathbb{E}|C_n|^2<\infty$ provided $\kappa$ large enough,
which implies that $C_n\rightarrow0$ a.s..

Finally, we  consider
$b_n\int_U\Psi_k\prod_{i=k}^{n-1}\omega_i\zeta_ik_adt$. Clearly,
\begin{equation}\label{LTE1.3.3}
\left|b_n\int_U\Psi_k\prod_{i=k}^{n-1}\omega_i\zeta_ik_adt\right|\leq\frac{Z_k(\mathbb{R})}{P_k}b_n\int_U\left|\prod_{i=k}^{n-1}\zeta_i\right|dt
\end{equation}
Since $\frac{Z_k(\mathbb{R})}{P_k}\rightarrow W$ a.s. as
$k\rightarrow\infty$, it remains to consider
$b_n\int_U|\prod_{i=k}^{n-1}\zeta_i|dt$. It suffices to show that
\begin{equation}\label{LTE1.3.4}
\limsup_{n\rightarrow\infty}
b_n\int_U\left|\prod_{i=k}^{n-1}\zeta_i(t)\right|dt\leq\limsup_{n\rightarrow\infty}
b_n\int_U\prod_{i=k}^{n-1}|\phi_i(t)|dt,
\end{equation}
and there exists a constant $\theta_1>0$ (not depending on $A$) such
that
\begin{equation}\label{LTE1.3.5}
\limsup_{n\rightarrow\infty}
b_n\int_U\prod_{i=k}^{n-1}|\phi_i(t)|dt\leq\int_{|t|\geq
A}e^{-\theta_1t^2}du\qquad \text{for any $A$}.
\end{equation}
Notice that
$$ b_n\int_U\left|\prod_{i=k}^{n-1}\zeta_i\right|dt\leq b_n\int_U\left|\prod_{i=k}^{n-1}\zeta_i-\prod_{i=k}^{n-1}\phi_i\right|dt+ b_n\int_U\prod_{i=k}^{n-1}|\phi_i|dt. $$
The proof of \cite{b} Lemma 3 gives
$|\zeta_i-\phi_i|\leq2(1-\tilde{m}_{i,k}/m_i)$, so we have
\begin{eqnarray*}
b_n\int_U\left|\prod_{i=k}^{n-1}\zeta_i-\prod_{i=k}^{n-1}\phi_i\right|dt&\leq &b_n\int_U\sum_{i=k}^{n-1}|\zeta_i- \phi_i|dt\\
&\leq&\frac{4}{a}b_n\sum_{i=k}^{n-1}\left(1-\frac{\tilde{m}_{i,k}}{m_i}\right)\\
&\leq
&C\sum_{i=k}^{n-1}(i+1)^{\alpha\gamma}\left(1-\frac{\tilde{m}_{i,k}}{m_i}\right)\rightarrow
0\quad as\;n\rightarrow\infty,
\end{eqnarray*}
provided $\alpha\gamma<\beta$.  Hence (\ref{LTE1.3.4}) holds. Now we turn to
prove (\ref{LTE1.3.5}). Split the set $U$ into two parts: $U_1=\{t:A/b_n\leq
t\leq\epsilon\}$ and $U_2=\{t:\epsilon\leq t\leq \frac{1}{a}\}$.
Since for some $\iota>0$, $|\phi_i(t)|\leq c_\iota<1$ for all
$|t|\geq\iota$, by Lemma \ref{LTL1.3.1}, we have  for all $|t|<\iota$,
$$|\phi_i(t)|\leq1-\frac{1-c_\iota^2}{8\iota^2}t^2\leq e^{-\gamma_1t^2},$$
where $\gamma_1=\frac{1-c_\iota^2}{8\iota^2}$. Thus
$$\sup_i\sup_{|t|\geq\epsilon} |\phi_i(t)| =\max\{e^{-\gamma_1\epsilon^2},c_\iota\}=:c'_\iota<1.$$
It follows that
\begin{equation}\label{LTE1.3.6}
b_n\int_{U_2}\prod_{i=k}^{n-1}|\phi_i(t)|dt\leq\frac{2}{a}b_n(c'_\iota)^{n-k-1}\rightarrow0\;as\;n\rightarrow\infty,
\end{equation}
and
\begin{equation}\label{LTE1.3.7}
b_n\int_{U_1}\prod_{i=k}^{n-1}|\phi_i(t)|dt\leq\int_{|t|\geq
A}\exp(-b_n^{-2}(n-k-1)\gamma_1t^2)dt.
\end{equation}
It is easy to see that
\begin{displaymath}
\lim_{n\rightarrow\infty}\frac{n-k-1}{b_n^2}=\left\{
\begin{array}{ll}\frac{1}{\theta^2}&\textrm{ if $\gamma=\frac{1}{2}$,}\\
\infty&\textrm{ if $0<\gamma<\frac{1}{2}$}.
\end{array}\right.
\end{displaymath}
So there exists a constant $\theta_1>0$ such that
$b_n^{-2}(n-k-1)\gamma_1>\theta_1$ for $n$ large enough. Thus
\begin{equation}\label{LTE1.3.8}
\limsup_{n\rightarrow\infty}
b_n\int_{U_1}\prod_{i=k}^{n-1}|\phi_i(t)|dt\leq\int_{|t|\geq
A}e^{-\theta_1t^2}dt\qquad \text{for any $A$.}
\end{equation}
Consequently, (\ref{LTE1.3.5}) holds via (\ref{LTE1.3.6}) and (\ref{LTE1.3.8}). This completes the proof.
\end{proof}

By a similar argument of Stone (1965, \cite{s}), we have the following Lemma.

\begin{lem}\label{LTL1.3.3}
If (\ref{den}) holds, then $\forall\varepsilon>0$, there exist
$n_0>0$ and $\delta>0$ such that $\forall
n\geq n_0$ and $\forall 0<h<\delta$,
\begin{equation}\label{cz}
h(Wp_L(x)-\varepsilon)\leq P_n^{-1}Z_n(b_n(x+a_n),b_n(x+a_n+h))\leq
h(Wp_L(x)+\varepsilon)\quad a.s., \quad\forall x\in\mathbb{R}.
\end{equation}
The null set can be taken to be independent of $x$.
\end{lem}

Now we turn to the proof of Theorem \ref{LTT1.2}:

\begin{proof}[Proof of Theorem \ref{LTT1.2}]
Fix $h>0$. $\forall\varepsilon>0$, take
	$0<\varepsilon'<\varepsilon/h$. By Lemmas  \ref{LTL1.3.2} and \ref{LTL1.3.3} , for this
$\varepsilon'>0$, there exist $n_0'>0$ and $\delta'>0$ such that
$\forall
n\geq n'_0$ and $\forall 0<h'<\delta'$,
\begin{equation*}
h'(Wp_L(x)-\varepsilon')\leq P_n^{-1}Z_n(b_n(x+a_n),b_n(x+a_n+h'))\leq
h'(Wp_L(x)+\varepsilon')\quad a.s., \quad\forall x\in\mathbb{R},
\end{equation*}
Let $h'=h/b_n$. Then there exist $\tilde{n}_0>0$ such that
$0<h'<\delta'$ for $n\geq\tilde{n}_0$. Take
$n_0:=\max\{n_0',\tilde{n}_0\}>0$, we have $\forall n\geq n_0$,
\begin{equation*}
h(Wp_L(x)-\varepsilon')\leq
b_nP_n^{-1}Z^n(b_n(x+a_n),b_n(x+a_n)+h)\leq
h(Wp_L(x)+\varepsilon')\quad a.s.,\; \forall x\in\mathbb{R},
\end{equation*}
which implies that
$$\sup_{x\in \mathbb{R}}|b_nP_n^{-1}Z_n(b_n(x+a_n),b_n(x+a_n)+h)-Whp_L(x)
|\leq\varepsilon'h<\varepsilon\qquad a.s.,$$
so that
$$\sup_{x\in \mathbb{R}}|b_nP_n^{-1}Z_n(x,x+h)-Whp_L(x/b_n-a_n)|<\varepsilon\qquad a.s..$$
The proof is finished.
\end{proof}

%%%%%%%%%%%%%%%%%%%%%%%%%%%       section   9    %%%%%%%%%%%%%%%%%%%%%%%%%%%%%%%%%%%%
\section{Central limit theorems for $ {\frac{\mathbb{E}_\xi Z_n(\cdot)}{\mathbb{E}_\xi Z_n(\mathbb{R})}}$, $ {\frac{\mathbb{E} Z_n(\cdot)}{\mathbb{E} Z_n(\mathbb{R})}}$ and $ {\mathbb{E}\frac{Z_n(\cdot)}{\mathbb{E}_\xi Z_n(\mathbb{R})}}$}\label{LTS9}
Now we return to consider the branching random walk with a random environment in time introduced in Section \ref{LTS1}.  When the environment $\xi$ is fixed, a branching random walk in  random environment is in fact a branching random walk in varying environment  introduced in Section \ref{LTS6}. We still    assume (\ref{LTE1.2}), which implies that
\begin{equation*}
\lim_{n\rightarrow\infty}\frac{1}{n}\log P_n= \mathbb{E}\log m_0>0\qquad\text{and}\qquad\lim_{n\rightarrow\infty}\frac{1}{n}\log
m_n=0\qquad a.s.
\end{equation*}
by the ergodic
theorem. Hence the assumption (\ref{LTE1.1.1}) is satisfied, so that (\ref{LTE1.1.2}) holds for some constant $c>1$ and  integer
$n_0=n_0(\xi)$ depending on $c$ and $\xi$.  Note that all the notations and results in Section \ref{LTS6} are still available under the quenched law $\mathbb{P}_\xi $ and the corresponding expectation $\mathbb{E}_\xi$.

Recall that $\nu_n(\cdot)=\frac{\mathbb{E}_\xi X_n(\cdot)}{m_n}$ is the intensity measure of $\frac{X_n}{m_n}$. Put
\begin{equation}
\mu_n=\int x\nu_n(dx)=\frac{1}{m_n}\mathbb E_\xi \sum_{i=1}^{N(u)}L_i(u)  \quad (u\in \mathbb T_n)
\end{equation}
and
\begin{equation}
\sigma_n^2=\int(x-\mu_n)^2\nu_n(dx)=\frac{1}{m_n}\mathbb E_\xi \sum_{i=1}^{N(u)}(L_i(u)-\mu_n)^2 \quad (u\in \mathbb T_n).
\end{equation}
We first have a central limit theorem for quenched means as follows.

\begin{thm}[Central limit theorem for quenched means $ \frac{\mathbb{E}_\xi Z_n(\cdot)}{\mathbb{E}_\xi Z_n(\mathbb{R})}$]\label{LTT1.1.3}  If $|\mu_0|<\infty\; a.s.$ and $ \mathbb{E}\sigma_0^2\in(0,\infty)$, then
$$\frac{ \mathbb{E}_\xi Z_n(-\infty,b_nx+a_n]}{ \mathbb{E}_\xi Z_n(\mathbb{R})}\rightarrow \Phi(x)\quad a.s.,$$
where $a_n=\sum_{i=0}^{n-1}\mu_i$ and $b_n=(\sum_{i=0}^{n-1}\sigma_i^2)^{1/2}$.
\end{thm}

\begin{proof}[Proof]%[Proof of Theorem \ref{LTT1.1.3}]
Notice that $\frac{ \mathbb{E}_\xi Z_n(\cdot)}{ \mathbb{E}_\xi Z_n(\mathbb{R})}=\nu_0*\cdots*\nu_{n-1}(\cdot)$.
It suffices to show that $\{\nu_n\}$ satisfies Lindeberg condition, i.e., for all $t>0$ ,
\begin{equation}\label{cq11}
\lim_{n\rightarrow\infty}\frac{1}{b_n^2}\sum_{i=0}^{n-1}\int_{|x-\mu_i|>tb_n}|x-\mu_i|^2\nu_i(dx)=0\qquad a.s..
\end{equation}
By the ergodic theorem,
\begin{equation}\label{LTE1.4.2}
\lim_{n\rightarrow\infty}\frac{b_n^2}{n}=\lim_{n\rightarrow\infty}\frac{1}{n}\sum_{i=0}^{n-1}\sigma_i^2=\mathbb{E}\sigma_0^2>0\qquad
a.s..
\end{equation}
So for a positive constant $a$ satisfying
$0<a^2<\mathbb{E}\sigma_0^2$, there exists an integer $n_0$ depending on $a$
and $\xi$ such that $b_n^2\geq a^2n$ for all $n\geq n_0$.  Fix a constant $ M>0$ . For $n\geq\max\{n_0, M\}$, we have
$b_n^2\geq a^2n\geq a^2 M$, so that
\begin{eqnarray*}
\frac{1}{b_n^2}\sum_{i=0}^{n-1}\int_ {|x-\mu_i|>tb_n}|x-\mu_i|^2\nu_i(dx)
\leq\frac{1}{a^2n}\sum_{i=0}^{n-1}\int_
{|x-\mu_i|>ta\sqrt{M}}|x-\mu_i|^2\nu_i(dx).
\end{eqnarray*}
Taking superior limit in the above inequality , we obtain
\begin{eqnarray*}
&&\limsup_{n\rightarrow\infty}\frac{1}{b_n^2}\sum_{i=0}^{n-1}\int_ {|x-\mu_i|>tb_n}|x-\mu_i|^2\nu_i(dx)\\
&\leq&\frac{1}{a^2}\lim_{n\rightarrow\infty}\frac{1}{n}\sum_{i=0}^{n-1}\int_
{|x-\mu_i|>ta\sqrt{M}}|x-\mu_i|^2\nu_i(dx)\\
&=&\frac{1}{a^2} \mathbb{E}\int_{|x-\mu_0|>ta\sqrt{M}}|x-\mu_0|^2\nu_0(dx).
\end{eqnarray*}
Let $M\rightarrow\infty$, it obvious that $ \mathbb{E}\int_{|x-\mu_0|>ta\sqrt{M}}|x-\mu_0|^2\nu_0(dx)\rightarrow0$ by the dominated convergence theorem,
since $\mathbb{E}\sigma_0^2<\infty$. This completes the proof.
\end{proof}

If the environment is $i.i.d.$, we can obtain a central limit theorem for annealed means.

\begin{thm}[Central limit theorem for annealed means $\frac{\mathbb{E} Z_n(\cdot)}{\mathbb{E} Z_n(\mathbb{R})}$]\label{LTT1.1.4} Assume that $\{\xi_n\}$ are $i.i.d.$.
Let $$\bar\mu =\frac{1}{\mathbb{E} m_0} \mathbb{E}\int x X_0(dx)=\frac{1}{\mathbb{E} m_0}{\mathbb E \sum\limits_{i=1}^NL_i}$$
and
$$\bar{\sigma}^2=\frac{1}{\mathbb{E} m_0} \mathbb{E}\int (x-\bar\mu)^2 X_0(dx)=\frac{1}{\mathbb{E} m_0}\mathbb E \sum\limits_{i=1}^N(L_i-\bar \mu)^2.$$  If $|\bar{\mu}|<\infty$ and $\bar{\sigma}^2\in(0,\infty)$, then
$$\frac{ \mathbb{E}Z_n(-\infty,\bar{b}_nx+\bar{a}_n]}{ \mathbb{E}Z_n(\mathbb{R})}\rightarrow \Phi(x),$$
where $\bar{a}_n=n\bar{\mu}$ and $\bar{b}_n=\sqrt{n}\bar{\sigma}$.
\end{thm}

\begin{proof}[Proof]%[Proof of Theorem \ref{LTT1.1.4}]
Denote $\bar\nu_n(\cdot)= \frac{ \mathbb{E}Z_n(\bar{b}_n\cdot+\bar{a}_n)}{ \mathbb{E}Z_n(\mathbb{R})}$. The characteristic function of
$\bar\nu_n$ is denoted by $\bar{\varphi}_n$. We can calculate
\begin{eqnarray*}
\bar{\varphi}_n(t)=\int e^{itx}\bar{\nu}_n(dx)
&=&(\mathbb{E} m_0)^{-n} \mathbb{E}\int e^{itx}Z_n(\bar{b}_ndx+\bar{a}_n)\\
&=&(\mathbb{E} m_0)^{-n}e^{-it\bar{a}_n/\bar{b}_n}\mathbb{E} \prod_{i=0}^{n-1} \mathbb{E}_\xi\int e^{itx/\bar{b}_n}X_n(dx)\\
&=&e^{-it\bar{a}_n/\bar{b}_n}\left(\frac{\mathbb{E} m_0(t/\bar{b}_n)}{\mathbb{E} m_0}\right)^n,
\end{eqnarray*}
where $m_n(t):= \mathbb{E}_\xi\int e^{itx}X_n(dx)$. The last step above is from the independency of $(\xi_n)$.
Denote $F(x)=\frac{\mathbb{E}X_0(x)}{\mathbb{E} m_0}$, then by the classic central limit theorem, we have
$$F^{*n}(\bar b_nx+\bar{a_n})\rightarrow \Phi(x).$$
Therefore,
$$\int e^{itx}F^{*n}(\bar b_ndx+\bar{a_n})\rightarrow g(t):=\int e^{itx}p(x)dx,$$
where $p(x)=\frac{1}{\sqrt{2\pi}}e^{-x^2/2}$ is the density function of standard normal distribution.
Notice that
\begin{eqnarray*}
\int e^{itx}F^{*n}(\bar b_ndx+\bar{a_n})
&=&e^{-it\bar{a}_n/\bar{b}_n}\int e^{ity/\bar{b}_n}F^{*n}(dy)\\
&=&e^{-it\bar{a}_n/\bar{b}_n}\left(\int e^{ity/\bar{b}_n}F(dy)\right)^n\\
&=&e^{-it\bar{a}_n/\bar{b}_n}\left(\frac{\mathbb{E} m_0(t/\bar{b}_n)}{\mathbb{E} m_0}\right)^n=\bar{\varphi}_n(t).
\end{eqnarray*}
We in fact have obtained
$\bar{\varphi}_n(t)\rightarrow g(t)$, it follows that $\bar{\nu}_n(x)\rightarrow \Phi(x)$  by the continuity theorem.
\end{proof}

By an argument similar to the proof of Theorem \ref{LTT1.1.4}, we obtain a central limit theorem as follows:

\begin{thm}[Central limit theorem for $\mathbb{E}\frac{Z_n(\cdot)}{ \mathbb{E}_\xi(\mathbb{R})}$]\label{LTT1.1.4'}  Assume that $\{\xi_n\}$ are $i.i.d.$.
Let $\bar{\mu}'= \mathbb{E}\int x \nu_0(dx)=\mathbb E\left(\frac{1}{m_0}\sum\limits_{i=1}^N L_i\right)$ and
$\bar{\sigma}'^2= \mathbb{E}\int (x-\bar\mu')^2 \nu_0(dx)=\mathbb E\left(\frac{1}{m_0}\sum\limits_{i=1}^N (L_i-\bar \mu')^2\right)$. If $|\bar{\mu}'|<\infty$ and $\bar{\sigma}'^2\in(0,\infty)$, then
$$\mathbb{E}\frac{ Z_n(-\infty,\bar{b}'_nx+\bar{a}'_n]}{ \mathbb{E}_\xi Z_n(\mathbb{R})}\rightarrow \Phi(x),$$
where $\bar{a}'_n=n\bar{\mu}'$ and $\bar{b}'_n=\sqrt{n}\bar{\sigma}'$.
\end{thm}

%%%%%%%%%%%%%%%%%%%%%%%%%%%%             section 10           %%%%%%%%%%%%%%%%%%%%%%%%%%%%%%%
\section{Central limit theorem and local limit theorem for $ {\frac{Z_n (\cdot)}{Z_n(\mathbb{R})}}$}\label{LTS10}
As we mentioned in last section (Section \ref{LTS9}), we can directly use the results of Theorems \ref{LTT1.1.1} and \ref{LTT1.2} considering the quenched law $\mathbb{P}_\xi$ and the corresponding expectation $\mathbb{E}_\xi$. However,
by the good properties of stationary and ergodic random process,
we have some similar but simper and more precise results than Theorems \ref{LTT1.1.1} and \ref{LTT1.2}.

\begin{thm}\label{LTT1.1.5}
Assume that for some $\varepsilon>0$,
$$\upsilon(\varepsilon):= \mathbb{E}\int|x|^{\varepsilon}\nu_0(dx)<\infty,$$
and $b_n=b_n(\xi)$ satisfying
$$b_n^{-1}=o(n^{-\gamma })\;a.s.\quad \text{for some} \;\;\gamma >0,$$
then
$$\Psi_n(t/b_n)-W\prod_{i=0}^{n-1}\phi_i(t/b_n)\rightarrow0\qquad a.s..$$
If in addition (A) holds wit
$\{a_n(\xi),b_n(\xi)\}$ and $g_\xi$, then
\begin{equation}\label{LTE1.1.15}
 e^{-ita_n}\Psi_n(t/b_n)\rightarrow g_{\xi}(t)W\qquad a.s.,
\end{equation}
 and for $x$ a continuity point of $L_{\xi}$,
\begin{equation*}
P_n^{-1}Z_n(-\infty,b_n(x+a_n)]\rightarrow
L_{\xi}(x)W\qquad a.s..
\end{equation*}
Moreover, (\ref{LTE1.1.15}) holds uniformly for $u$ in compact sets.
\end{thm}

The following result is the most important central limit theorem of this paper.

\begin{thm}[Central limit theorem for $\frac{Z_n(\cdot)}{Z_n(\mathbb{R})}$]\label{LTT1.1.6}
 If $\mathbb{E}|\mu_0|^\varepsilon<\infty$ for some $\varepsilon>0$ and $\mathbb{E}\sigma_0^2\in(0,\infty)$, then
 \begin{equation}\label{LTET10.2}
\frac{ Z_n(-\infty,b_nx+a_n]}{ Z_n(\mathbb{R})}\rightarrow  \Phi(x)\quad a.s.\; on\;\{Z_n(\mathbb{R})\rightarrow\infty\},
\end{equation}
  where $a_n=\sum_{i=0}^{n-1}\mu_i$ and $b_n=(\sum_{i=0}^{n-1}\sigma_i^2)^{1/2}$.
\end{thm}

\noindent\emph{Remark.} If $\mathbb{E}\int x^2\nu_0(dx)<\infty$,  it can be easily seen that $\mathbb{E}\mu_0^2<\infty$ and $\mathbb{E}\sigma_0^2<\infty$.
\\*

Theorem \ref{LTT1.1.6} is an extension of the results of Kaplan and Asmussen (1976, II, Theorem 1) and Biggins (1990) on deterministic branching random walks.
\\*

Similarly to the case of varying environment, we also have the local
limit theorems corresponding to Theorems \ref{LTT1.1.5} and    \ref{LTT1.1.6}
respectively.

\begin{thm}\label{LTT1.1.8}
Assume that $\nu_0$ is non-lattice a.s., (A) holds  with $\{a_n(\xi),b_n(\xi)\}$ satisfying $b_n\sim \theta n^\gamma
a.s.$  for some constants $0<\gamma\leq\frac{1}{2}$ and $\theta>0$,  and
$g_\xi$ is integrable. If $\upsilon(\varepsilon)<\infty$ for some
$\varepsilon>0$, and
\begin{equation}\label{LTET1.1.81}
\mathbb{E}\frac{N}{m_0}(\log^+ N)^{1+\beta}<\infty
\end{equation}
for some $\beta>\gamma$, then $\forall h>0$,
$$\sup_{x\in\mathbb{R}}|b_nP_n^{-1}Z_n(x,x+h)-Whp_L(x/b_n-a_n)|\rightarrow0\qquad a.s. ,$$
where $p_L$ is the density function of $L_\xi$.
\end{thm}

Theorem \ref{LTT1.1.9} below is a direct  consequence of Theorem \ref{LTT1.1.8}.  To verify the conditions of Theorem \ref{LTT1.1.8}, see the proof of Theorem \ref{LTT1.1.6}.

\begin{thm}[Local limit theorem for $\frac{Z_n(\cdot)}{Z_n(\mathbb{R})}$]\label{LTT1.1.9}
Assume that $\nu_0$ is non-lattice a.s.. If $\mathbb{E}|\mu_0|^\varepsilon<\infty$ for some $\varepsilon>0$,  $\mathbb{E}\sigma_0^2\in(0,\infty)$, and
$$\mathbb{E}\frac{N}{m_0}(\log^+ N)^{\beta}<\infty$$
for some $\beta>\frac{3}{2}$,
then $\forall h>0$,
$$\sup_x|b_n\frac{Z_n(x,x+h)}{Z_n(\mathbb{R})} -hp(\frac{x-a_n}{b_n})|\rightarrow0\quad    a.s.\; on\;\{Z_n(\mathbb{R})\rightarrow\infty\},$$
where $a_n=\sum_{i=0}^{n-1}\mu_i$, $b_n=(\sum_{i=0}^{n-1}\sigma_i^2)^{1/2}$, and
$p(x)=\frac{1}{\sqrt{2\pi}}e^{-x^2/2}$ is the density function of standard normal distribution.
\end{thm}

For the deterministic environment case,  similar result was showed by Biggins (1990).
\\*

From Theorem \ref{LTT1.1.9}, we immediately obtain the following corollary.

\begin{co}\label{BRCC1}
Under the conditions of Theorem \ref{LTT1.1.9},   we have $\forall a<b$,
$$b_n\frac{Z_n(a+a_n,b+a_n)}{Z_n(\mathbb{R})} \rightarrow \frac{1}{\sqrt{2\pi}}(b-a)\quad   a.s.\; on\;\{Z_n(\mathbb{R})\rightarrow\infty\},$$
where $a_n=\sum_{i=0}^{n-1}\mu_i$ and $b_n=(\sum_{i=0}^{n-1}\sigma_i^2)^{1/2}$.
\end{co}

Corollary \ref{BRCC1} coincide with a result of Kaplan and Asmussen (1976, II, Theorem 2) on  deterministic branching random walks.

%%%%%%%%%%%%%%%%%%%%%%%%%%%%%%%%       section   11      %%%%%%%%%%%%%%%%%%%%%%%%%%%%%%%%%%%%%%%%%%%
\section{Proofs of Theorems \ref{LTT1.1.5}-\ref{LTT1.1.8}}\label{LTS11}
Before of the proof of Theorem \ref{LTT1.1.5}, we prove a lemma at first.
\begin{lem}\label{LTL1.4.1}
Let $\beta\geq0$.
If $\mathbb{E}\frac{N_0}{m_0}(\log^+N_0)^{1+\beta}<\infty$, then for all
$\kappa$, $\sum _nn^{\beta}(1-\tilde{m}_{n,\kappa}/m_n)<\infty$ a.s..
\end{lem}

\begin{proof}[Proof]
As the proof of Lemma \ref{LTL1.2.1}, we have
$$\sum_n\left(1-\frac{\tilde{m}_{n,\kappa}}{m_n}\right)=\sum_n\frac{1}{m_n}\mathbb{E}_{\xi}N_nI_n^c(N_n).$$
By (\ref{LTE1.1.2}), for n large enough,
$$ \mathbb{E}_{\xi}N_nI_n^c(N_n)\leq \mathbb{E}_{\xi}N_n {1}_{\{N_n(\log N_n)^\kappa>c^{n+1}\}}.$$
Taking expectation for the series
$\sum\frac{n^{\beta}}{m_n}\mathbb{E}_{\xi}N_n {1}_{\{N_n(\log
N_n)^\kappa>c^{n+1}\}}$, we have
\begin{eqnarray*}
&&\mathbb{E}\left(\sum_n\frac{n^{\beta}}{m_n}\mathbb{E}_{\xi}N_n {1}_{\{N_n(\log
N_n)^\kappa>c^{n+1}\}}\right)\\
&=&\sum_nn^{\beta}\mathbb{E}\frac{N_0}{m_0} {1}_{\{N_0(\log
N_0)^\kappa>c^{n+1}\}}\\
&=&\mathbb{E}\frac{N_0}{m_0}\sum_nn^{\beta} {1}_{\{N_0(\log
N_0)^\kappa>c^{n+1}\}}\\
&\leq&C\mathbb{E}\frac{N_0}{m_0}(\log^+N_0)^{1+\beta}<\infty,
\end{eqnarray*}
so that $\sum_n n^{\beta}(1-\tilde{m}_{n,\kappa}/m_n)<\infty$ a.s..
\end{proof}

\begin{proof}[Proof of Theorem \ref{LTT1.1.5}]
From the proof of Theorem \ref{LTT1.1.1}, we know
that in fact, instead of (\ref{LTE1.1.3}), we only need (\ref{LTE1.2.6}) and
$\sum_n(1-\tilde{m}_{n,\kappa}/m_n)<\infty$ for the suitable $\kappa$.
For the branching random walk in a stationary and ergodic random
environment, Lemma \ref{LTL1.4.1} tells us  that the condition
$\mathbb{E}\frac{N_0}{m_0}\log^+N_0<\infty$ ensures
$\sum_n(1-\tilde{m}_{n,\kappa}/m_n)<\infty$. And it also ensures
(\ref{LTE1.2.6}), since for any $\delta_1>0$,
$$\mathbb{E}\left(\sum_n\frac{1}{m_nn^{1+\delta_1}}\mathbb{E}_{\xi}N_n\log^+N_n\right)=\sum_n\frac{1}{n^{1+\delta_1}}\mathbb{E}\frac{N_0}{m_0}\log^+N_0<\infty.$$
By the ergodic theorem,
$$\lim_n\frac{\upsilon_n(\varepsilon)}{n}=\upsilon(\varepsilon)<\infty\qquad a.s..$$
Hence the condition (\ref{LTE1.1.41}) holds. Thus Theorem \ref{LTT1.1.5} is just a direct
consequence of Theorem \ref{LTT1.1.1}.
\end{proof}

\begin{proof}[Proof of Theorem \ref{LTT1.1.6}]
We will use Theorem \ref{LTT1.1.5} to prove Theorem \ref{LTT1.1.6}. Assume that $0<\varepsilon\leq 2$ (otherwise, consider $\min\{\varepsilon,2\}$ instead of
$\varepsilon$), then
$$\upsilon(\varepsilon)= \mathbb{E}\int|x|^\varepsilon\nu_0(dx)\leq
C_\varepsilon\left( \mathbb{E}\int|x-\mu_0|^\varepsilon\nu_0(dx)+\mathbb{E}|\mu_0|^\varepsilon\right)<\infty.$$
By (\ref{LTE1.4.2}), $b_n\sim \mathbb{E}\sigma_0^2\sqrt{n}$ a.s., which implies that for
any $0<\gamma<\frac{1}{2}$, $b_n^{-1}=o(n^{-\gamma})a.s.$. The proof of Theorem \ref{LTT1.1.3} show that $\{\nu_n\}$ satisfies
Lindeberg condition, so that  (A) holds with  $a_n'=a_n/b_n$ and $b_n'=b_n$.
By Theorem \ref{LTT1.1.5},
$$P_n^{-1}Z_n(-\infty, b_nx+a_n]\rightarrow \Phi(x)W\quad a.s.$$
Notice that $Z_n(\mathbb{R})/P_n\rightarrow W$ a.s. and $\mathbb{P}(W>0)=\mathbb{P}(Z_n(\mathbb{R})\rightarrow\infty)$. Thus (\ref{LTET10.2}) holds.
 \end{proof}

\begin{lem}\label{LTL1.4.2}
Let $A>0$  be a constant. Assume that $b_n\sim \theta n^\gamma
a.s.$  for some constants $0<\gamma\leq\frac{1}{2}$.
If $\nu_0$ is non-lattice a.s.,then there exists a constant
$\theta_1>0$ (not depending on $A$) such that
\begin{equation}\label{LTE1.4.4}
\limsup_{n\rightarrow\infty}
b_n\int_U\prod_{i=k}^{n-1}|\phi_i(t)|dt\leq\int_{|t|\geq
A}e^{-\theta_1t^2}dt\quad a.s.,
\end{equation}
where $k=J(n)$ the same as the proof of Theorem \ref{LTT1.1.1} and
$U=\{t:\frac{A}{b_n}\leq|t|\leq\frac{1}{a}\}$.
\end{lem}

\begin{proof}[Proof]
Take $0<2\epsilon<\frac{1}{a}$. Like the last part
of the proof of Theorem \ref{LTT1.2}, split $U$ into $U_1$ and $U_2$, so
$$b_n\int_U\prod_{i=k}^{n-1}|\phi_i(t)|dt=b_n\int_{U_1}\prod_{i=k}^{n-1}|\phi_i(t)|dt+b_n\int_{U_2}\prod_{i=k}^{n-1}|\phi_i(t)|dt.$$
Since $\nu_i$ is non-lattice a.s., we have
\begin{equation}\label{LTE1.4.5}
\sup_{\epsilon\leq|t|\leq
a^{-1}}|\phi_i(t)|=:c_i(\epsilon,a)=c_i<1\qquad a.s..
\end{equation}
Hence by Lemma \ref{LTL1.3.1}, for $|t|<\epsilon$,
\begin{equation}\label{LTE1.4.6}
|\phi_i(t)|\leq1-\frac{1-c_i^2}{8\epsilon^2}t^2\leq\exp(-\frac{1-c_i^2}{8\epsilon^2}t^2)=e^{-\alpha_it^2}\qquad
a.s.,
\end{equation}
where $\alpha_i=\frac{1-c_i^2}{8\epsilon^2}>0$ a.s.. Using (\ref{LTE1.4.5}),
we immediately get
\begin{equation}\label{LTE1.4.7}
b_n\int_{U_2}\prod_{i=k}^{n-1}|\phi_i(t)|\leq\frac{2}{a}b_n\prod_{i=k}^{n-1}c_i\rightarrow0\qquad
a.s.,
\end{equation}
since
$$\lim_{n\rightarrow\infty}\frac{\log b_n+\sum_{i=k}^{n-1}\log c_i}{n}=\mathbb{E}\log c_0<0\qquad a.s..$$
Observe that
\begin{displaymath}
\lim_{n\rightarrow\infty}\frac{\sum_{i=k}^{n-1}\alpha_i}{b_n^2}=\left\{
\begin{array}{ll}\frac{1}{\theta^2}\mathbb{E}\alpha_0>0&\textrm{ if $\gamma=\frac{1}{2}$}\\
\infty&\textrm{ if $0<\gamma<\frac{1}{2}$}
\end{array}\right.\qquad a.s..
\end{displaymath}
Take $0<\theta_1<\frac{1}{\theta^2}\mathbb{E}\alpha_0$. Using (\ref{LTE1.4.6}), we have for
$n$ large,
\begin{equation}\label{LTE1.4.8}
b_n\int_{U_1}\prod_{i=k}^{n-1}|\phi_i(t)|dt\leq\int_{|t|\geq
A}\exp(-b_n^{-2}\sum_{i=k}^{n-1}\alpha_it^2)du\leq\int_{|t|\geq
A}e^{-\theta_1t^2}dt\qquad a.s..
\end{equation}
(\ref{LTE1.4.7}) and (\ref{LTE1.4.8}) yield (\ref{LTE1.4.4}).
\end{proof}

\begin{proof}[Proof of Theorem \ref{LTT1.1.8}]
In the proof of Lemma \ref{LTL1.3.2}, the condition
(\ref{LTET1.1.21}) is just used to ensure (\ref{LTE1.3.5}) (i.e.(\ref{LTE1.4.4}) in random
environment), which always holds in random environment if $\nu_i$ is
non-lattice a.s., by Lemma \ref{LTL1.4.2}. Besides,
\begin{eqnarray*}
 \mathbb{E}\left(\sum_n\frac{1}{m_nn(\log n)^{1+\delta}} \mathbb{E}_\xi
N_n(\log^+N_n)^{1+\beta}\right)
=\sum_n\frac{1}{n(\log n)^{1+\delta}}\mathbb{E}\frac{N_0}{m_0}(\log^+N_0)^{1+\beta}<\infty.
\end{eqnarray*}
So (\ref{LTET1.1.81}) implies (\ref{LTET1.1.22}). Theorem \ref{LTT1.1.8} is a consequence of Theorem \ref{LTT1.2}.
\end{proof}

%%%%%%%%%%%%%%%%%%%%%%%%%%%%%%%%     section   12       %%%%%%%%%%%%%%%%%%%%%%%%%%%%%%%%%%%%%%
\section{Central limit theorems for $ {\mathbb{E}_\xi\frac{Z_n(\cdot)}{Z_n(\mathbb{R})} }$  and  $ \mathbb{E}\frac{Z_n(\cdot)}{Z_n(\mathbb{R})} $}\label{LTS12}
From Theorem  \ref{LTT1.1.6}, it is not hard to obtain the following  central  limit theorems for the probability measures $\mathbb{E}_\xi(\frac{Z_n(\cdot)}{Z_n(\mathbb{R})}|Z_n(\mathbb{R})>0)$ and $\mathbb{E}(\frac{Z_n(\cdot)}{Z_n(\mathbb{R})}|Z_n(\mathbb{R})>0)$:

\begin{thm}[Central limit theorems for $ {\mathbb{E}_\xi\frac{Z_n(\cdot)}{Z_n(\mathbb{R})} }$  and  $ \mathbb{E}\frac{Z_n(\cdot)}{Z_n(\mathbb{R})} $]\label{LTT1.1.7}
 If $\mathbb{E}|\mu_0|^\varepsilon<\infty$ for some $\varepsilon>0$ and $\mathbb{E}\sigma_0^2\in(0,\infty)$, then
\begin{eqnarray}
&&  \mathbb{E}_\xi\left(\left.\frac{Z_n(-\infty, b_nx+a_n]}{Z_n(\mathbb{R})}\right|Z_n(\mathbb{R})>0\right)\rightarrow \Phi(x)\qquad a.s., \label{LTET1.1.7a}\\
&&\mathbb{E}\left(\left.\frac{Z_n(-\infty,b_nx+a_n]}{Z_n(\mathbb{R})}\right|Z_n(\mathbb{R})>0\right)\rightarrow  \Phi(x),\label{LTET1.1.7b}
\end{eqnarray}
where $a_n=\sum_{i=0}^{n-1}\mu_i$ and $b_n=(\sum_{i=0}^{n-1}\sigma_i^2)^{1/2}$.
 \end{thm}

\begin{proof}[Proof]%[Proof of Theorem \ref{LTT1.1.7}]
Theorem \ref{LTT1.1.7}  is a consequence of Theorem \ref{LTT1.1.6}.  We only prove (\ref{LTET1.1.7b}), the proof for  (\ref{LTET1.1.7a}) is similar.
By Theorem \ref{LTT1.1.6},
\begin{equation}\label{cq3}
\left(\frac{Z_n(-\infty,b_n x+a_n]}{Z_n(\mathbb{R})}-\Phi(x)\right) {1}_{\{Z_n(\mathbb{R})\rightarrow\infty\}}
\rightarrow 0\quad a.s..
\end{equation}
The condition $\mathbb{E}\frac{N}{m_0}\log^+N<\infty$ ensures that
$$\lim_{n\rightarrow\infty}\mathbb{P}(Z_n(\mathbb{R})>0)=\mathbb{P}(Z_n(\mathbb{R})\rightarrow\infty)>0.$$
Observing that
\begin{eqnarray*}
&&\left|\mathbb{E}\left(\left.\frac{Z_n(-\infty,b_n x+a_n]}{Z_n(\mathbb{R})}\right|Z_n(\mathbb{R})>0\right)-\Phi(x) \right|\\
&=&\frac{1}{\mathbb{P}(Z_n(\mathbb{R})>0)}\left|\mathbb{E} {1}_{\{Z_n(\mathbb{R})>0\}}\left(\frac{Z_n(-\infty,b_n x+a_n]}{Z_n(\mathbb{R})}-\Phi(x)\right)\right|\\
&\leq&\frac{1}{\mathbb{P}(Z_n(\mathbb{R})>0)}\left|\mathbb{E}\left( {1}_{\{Z_n(\mathbb{R})>0\}}- {1}_{\{Z_n(\mathbb{R})\rightarrow\infty\}}\right)
\left(\frac{Z_n(-\infty,b_n x+a_n]}{Z_n(\mathbb{R})}-\Phi(x)\right)\right|\\
&&+\frac{1}{\mathbb{P}(Z_n(\mathbb{R})>0)}\left|\mathbb{E} {1}_{\{Z_n(\mathbb{R})\rightarrow\infty\}}\left(\frac{Z_n(-\infty,b_n x+a_n]}{Z_n(\mathbb{R})}-\Phi(x)\right)\right|,
\end{eqnarray*}
we only need to show that the two terms in the right side of the inequality above tend to zero as $n$ tends to infinity.  Since
$$0\leq\frac{Z_n(-\infty,b_n x+a_n]}{Z_n(\mathbb{R})}\leq 1\quad\text{and}\quad 0\leq\Phi(x)\leq1 ,$$
we have
$$\left|\frac{Z_n(-\infty,b_n x+a_n]}{Z_n(\mathbb{R})}-\Phi(x)\right|\leq 1.$$
Notice (\ref{cq3}), by the dominated convergence theorem, we get
$$\left|\mathbb{E} {1}_{\{Z_n(\mathbb{R})\rightarrow\infty\}}\left(\frac{Z_n(Z_n(-\infty,b_n x+a_n])}{Z_n(\mathbb{R})}-\Phi(x)\right)\right|
\rightarrow 0.$$
For the first term, we have
\begin{eqnarray*}
&&\left|\mathbb{E}( {1}_{\{Z_n(\mathbb{R})>0\}}- {1}_{\{Z_n(\mathbb{R})\rightarrow\infty\}})
\left(\frac{Z_n(-\infty,b_n x+a_n]}{Z_n(\mathbb{R})}-\Phi(x)\right)\right|\\
&\leq&\mathbb{E}| {1}_{\{Z_n(\mathbb{R})>0\}}- {1}_{\{Z_n(\mathcal{R})\rightarrow\infty\}}|\\
&=&\mathbb{P}(Z_n(\mathbb{R})>0)-\mathbb{P}(Z_n(\mathbb{R})\rightarrow\infty)\rightarrow 0.
\end{eqnarray*}
This completes the proof.
\end{proof}


\begin{thebibliography}{99}

\bibitem{a1}K. B. Athreya, S. Karlin,   On branching processes in random environments I \& II.  Ann. Math.  Statist. 42 (1971), 1499-1520 \& 1843-1858.

\bibitem{a}K. B. Athreya, P. E. Ney,  Branching Processes. Springer, Berlin, 1972.

\bibitem{biggins77}J. D. Biggins, Martingale convergence in the branching random walk. J. Appl. prob. 14 (1977), 25-37.

\bibitem{biggins}J. D. Biggins,  Chernoff's theorem in the branching random walk. J. Appl. Probab. 14 (1977), 630-636.

\bibitem{biggins79}J. D. Biggins, Growth rates in the branching random walk. Z. Wahrsch. verw. Geb. 48 (1979), 17-34.

\bibitem{biggins97}J. D. Biggins, A. E. Kyprianou, Seneta-Heyde norming in the branching random walk. Ann. Prob. 25 (1997), 337-360.

\bibitem{b}J. D. Biggins, The central limit theorem for the
supercritical branching random walk,and related results.  Stoch.
Proc. Appl. 34 (1990), 255-274.

\bibitem{biggins04}J. D. Biggins, A. E. Kyprianou, Measure change in multitype
branching. Adv. Appl. Probab. 36 (2004), 544-581.

\bibitem{chauvin}B. Chauvin, A. Rouault, Boltzmann-Gibbs weights in the branching random walk.   In K.B. Athreya, P. Jagers,  (eds.), Classical and Modern Branching Processes, IMA Vol. Math. Appl. 84, pp. 41-50,  Springer-Verlag, New York, 1997.

\bibitem{c}H. Cram\'er, Random variables and probability distributions. Cambridge University Press, Cambridge, 1937.

\bibitem{z}A. Dembo,  O. Zeitouni, Large deviations Techniques and Applications. Springer, New York, 1998.

\bibitem{durrett}R. Durrett, T. Liggett, Fixed points  of the smoothing transformation. Z. Wahrsch. verw. Geb.  64 (1983), 275-301.

\bibitem{fel}W. Feller, An introduction to probability theory and its applications, Vol.II. Wiley, New York, 1971.

\bibitem{franchi}J. Franchi,  Chaos Multiplicatif: un Traitement Simple et Complet de la Fonction de Partition.  S\'eminaire de probabilit\'es XXIX. Springer, 1995, 194-201.

\bibitem{guivarch}Y. Guivarc'h, Sur une extension de la notion de loi semi-stable. Ann. Inst. H. Poincar\'e. Probab. Statist. 26 (1990), 261-285.

\bibitem{Harris}T. E. Harris, The theory of branching process. Springer, Berlin, 1963.

\bibitem{ham}B. Hambly,  On the limit distribution of a supercritical
branching process in a random environment. J. Appl. Prob. 29 (1992), 499-518.

\bibitem{huang}C. Huang, Q. Liu, Moments, moderate  and large deviations for a branching process
in a random environment. Stoch. Proc. Appl. 122 (2012), 522-545.

 \bibitem{jajer}P. Jagers, Galton-Watson processes in varying environments. J. Appl. Prob. 11 (1974), 174-178.

 \bibitem{kahane}J. P. Kahane, J. Peyri\`ere, Sur certaines martingales de Benoit Mandelbrot.  Adv. Math. 22 (1976), 131-145.

\bibitem{ka}N. Kaplan, S. Asmussen, Branching random walks I \& II. Stoch. Proc. Appl. 4 (1976), 1-13 \& 15-31.

\bibitem{ku}D. Kuhlbusch, On weighted branching processes in random environment. Stoch. Proc. Appl. 109 (2004), 113-144.

\bibitem{k}C. F. Klebaner, Branching random walk in varying environment. Adv. Appl. Proba. 14 (1982), 359-367.

\bibitem{liu97}Q. Liu, Sur une \'equation fonctionnelle et ses applications: une extension du th\'eor\`eme de Kesten-Stigum concernant des processus de branchement. Adv. Appl. Prob. 29 (1997), 353-373.

\bibitem{liu98}Q. Liu, Fixed points of a generalized smoothing transformation and applications to branching processes. Adv. Appl. Prob. 30 (1998), 85-112.

\bibitem{liu00}Q. Liu, On generalized multiplicative cascades. Stoch. Proc. Appl. 86 (2000), 61-87.

\bibitem{liu4}Q. Liu, A. Rouault,  Limit theorems for Mandelbrot's multiplicative cascades. Ann. Appl. Proba. 10 (2000), 218-239.

\bibitem{liu01}Q. Liu, Asymptotic properties absolute continuity of laws stable by random weighted mean. Stoch. Proc. Appl. 95 (2001), 83-107.

\bibitem{liu5}Q. Liu, Branching Random Walks in Random Environment. Proceedings
of the 4th International Congress of Chinese Mathematicians, 2007 (ICCM 2007), Vol. II, 702-719. Eds: L. Ji, K. Liu, L. Yang, S.-T. Yau.




\bibitem{lyons}R. Lyons, A simple path to Biggins's martingale convergence for branching random walk. In K.B. Athreya, P. Jagers,  (eds.),  Classical and Modern Branching Processes, IMA Vol. Math. Appl. 84, pp. 217-221,  Springer-Verlag, New York, 1997.

 \bibitem{smith}W. L. Smith,  W. Wilkinson,  On branching processes in random environments. Ann. Math. Statist. 40 (1969), 814-827.

 \bibitem{stam}A. J. Stam, On a conjecture of Harris, Z. Wahrsch. Verw. Geb. 5 (1966) 202-206.

 \bibitem{s}C. Stone, A local limit theorem for nonlattice multi-dimensional distribution functions. Ann. Math. Statist. 36 (1965), 546-551.

\bibitem{tanny1}D. Tanny,  Limit theorems for branching processes in
a random environment. Ann. Proba.  5 (1977), 100-116.

\bibitem{tanny2}D. Tanny, A necessary and sufficient condition for a
branching process in a random environment to grow like the product
of its means. Stoch. Proc. Appl. 28 (1988), 123-139.




\end{thebibliography}
\end{document}